\documentclass[11pt]{article}
\usepackage{palatino}
\usepackage{amsmath}
\usepackage{amsthm}
\usepackage{amssymb}
\usepackage{url}
\usepackage{latexsym}
\usepackage[xcolor=pst]{pstricks}
\usepackage{graphicx, pst-plot, pst-node, pst-text, pst-tree}
\usepackage{titlefoot}
\usepackage[small]{titlesec}
\usepackage{units} 
\usepackage[small,it]{caption}

\usepackage{color}
\definecolor{lightgray}{rgb}{0.8, 0.8, 0.8}
\definecolor{darkgray}{rgb}{0.7, 0.7, 0.7}
\definecolor{darkblue}{rgb}{0, 0, .4}

\usepackage[bookmarks]{hyperref}
\hypersetup{
        colorlinks=true,
        linkcolor=darkblue,
        anchorcolor=darkblue,
        citecolor=darkblue,
        urlcolor=darkblue,
        pdfpagemode=UseThumbs,
        pdftitle={Geometric grid classes of permutations},
        pdfsubject={Combinatorics},
        pdfauthor={Albert, Atkinson, Bouvel, Ruskuc, and Vatter},
}

\newcounter{todocounter}

\theoremstyle{plain}
\newtheorem{theorem}{Theorem}[section]
\newtheorem{proposition}[theorem]{Proposition}

\newtheorem{corollary}[theorem]{Corollary}

\theoremstyle{definition}

\makeatletter
\newfont{\footsc}{cmcsc10 at 8truept}
\newfont{\footbf}{cmbx10 at 8truept}
\newfont{\footrm}{cmr10 at 10truept}
\renewenvironment{abstract}%
                {
                  \begin{list}{}%
                     {\setlength{\rightmargin}{1in}%
                      \setlength{\leftmargin}{1in}}%
                   \item[]\ignorespaces\begin{small}}%
                 {\end{small}\unskip\end{list}}

\newcommand{\Av}{\operatorname{Av}}
\newcommand{\Age}{\operatorname{Age}}

\newcommand{\C}{\mathcal{C}}
\newcommand{\D}{\mathcal{D}}
\newcommand{\E}{\mathcal{E}}

\newcommand{\zpm}{0/\mathord{\pm} 1}
\newcommand{\cell}[2]{a_{#1#2}}
\newcommand{\zdpm}{0/\mathord{\bullet}/\mathord{\pm} 1} 

\newcommand{\Grid}{\operatorname{Grid}}
\newcommand{\Geom}{\operatorname{Geom}}
\newcommand{\Sub}{\operatorname{Sub}}
\newcommand{\st}{\::\:}

\newcommand{\equivfig}{\approx}
\newcommand{\bij}{\varphi}
\newcommand{\opinj}{\phi}
\newcommand{\gridded}{\sharp}

\newcommand{\gridlex}{\sqsubset}
\newcommand{\gridlexeq}{\sqsubseteq}
\newcommand{\emptyword}{\varepsilon}
\newcommand{\fnmatrix}[2]{\mbox{\begin{footnotesize}$\left(\begin{array}{#1}#2\end{array}\right)$\end{footnotesize}}}

\newrgbcolor{gray69}{0.7 0.7 0.7}
\newrgbcolor{gray53}{0.53 0.53 0.53}
\newrgbcolor{gray80}{0.8 0.8 0.8}
\newrgbcolor{gray90}{0.90 0.90 0.90}
\newrgbcolor{gray95}{0.95 0.95 0.95}

\datefoot{\today}
\amssubj{05A05, 05A15}

\newpagestyle{main}[\small]{
        \headrule
        \sethead[\usepage][][]
        {\sc Geometric Grid Classes of Permutations}{}{\usepage}}

\title{\sc Geometric Grid Classes of Permutations}
\author{%
Michael H. Albert\\[-0.25ex]
\small Department of Computer Science\\[-0.5ex]
\small University of Otago\\[-0.5ex]
\small Dunedin, New Zealand\\[1.5ex]
M. D. Atkinson\\[-0.25ex]
\small Department of Computer Science\\[-0.5ex]
\small University of Otago\\[-0.5ex]
\small Dunedin, New Zealand\\[1.5ex]
Mathilde Bouvel\\[-0.25ex]
\small CNRS, LaBRI\\[-0.5ex]
\small Universit\'e Bordeaux 1\\[-0.5ex]
\small Bordeaux, France\\[1.5ex]
Nik Ru\v{s}kuc\\[-0.25ex]
\small School of Mathematics and Statistics\\[-0.5ex]
\small University of St Andrews\\[-0.5ex]
\small St Andrews, Scotland\\[1.5ex]
Vincent Vatter\\[-0.25ex]
\small Department of Mathematics\\[-0.5ex]
\small University of Florida\\[-0.5ex]
\small Gainesville, Florida USA\\[-1.5ex]
}

\titleformat{\section}
        {\large\sc}
        {\thesection.}{1em}{}   

\date{}

\begin{document}
\maketitle

\pagestyle{main}

\begin{abstract}
A geometric grid class consists of those permutations that can be drawn on a specified set of line segments of slope $\pm 1$ arranged in a rectangular pattern governed by a matrix.  Using a mixture of geometric and language theoretic methods, we prove that such classes are specified by finite sets of forbidden permutations, are partially well ordered, and have rational generating functions.  Furthermore, we show that these properties are inherited by the subclasses (under permutation involvement) of such classes, and establish the basic lattice theoretic properties of the collection of all such subclasses.
\end{abstract}

\section{Introduction}

Subsequent to the resolution of the Stanley-Wilf Conjecture, in 2004 by Marcus and Tardos~\cite{marcus:excluded-permut:}, two major research programs have emerged in the study of permutation classes:
\begin{itemize}
\item to characterise the possible growth rates of permutation classes, and
\item to provide necessary and sufficient conditions for permutation classes to have amenable generating functions.
\end{itemize}
With regard to the first program, we point to the work of Kaiser and Klazar~\cite{kaiser:on-growth-rates:}, who characterised the possible growth rates up to $2$, and the work of Vatter~\cite{vatter:small-permutati:}, which extends this characterisation up to the algebraic number $\kappa\approx 2.20557$, the point at which infinite antichains begin to emerge, and where the transition from countably many to uncountably many permutation classes occurs.  The second programme is illustrated by the work of Albert, Atkinson, and Vatter~\cite{albert:subclasses-of-t:}, who showed that all subclasses of the separable permutations not containing $\Av(231)$ or a symmetry of this class have rational generating functions.

Both research programs rely on structural descriptions of permutation classes, in particular, the notion of grid classes.  Here we study the enumerative and order-theoretic properties of a certain type of grid classes called geometric grid classes.

While we present formal definitions in the next two sections, geometric grid classes may be defined briefly as follows.  Suppose that $M$ is a $\zpm$ matrix.  The \emph{standard figure} of $M$, which we typically denote by $\Lambda$, is the point set in $\mathbb{R}^2$ consisting of:
\begin{itemize}
\item the increasing open line segment from $(k-1,\ell-1)$ to $(k,\ell)$ if $M_{k,\ell}=1$ or
\item the decreasing open line segment from $(k-1,\ell)$ to $(k,\ell-1)$ if $M_{k,\ell}=-1$.
\end{itemize}
(Note that in order to simplify this correspondence, we index matrices first by column, counting left to right, and then by row, counting bottom to top throughout.)  The \emph{geometric grid class} of $M$, denoted by $\Geom(M)$, is then the set of all permutations that can be drawn on this figure in the following manner.  Choose $n$ points in the figure, no two on a common horizontal or vertical line.  Then label the points from $1$ to $n$ from bottom to top and record these labels reading left to right.

Much of our perspective and inspiration comes from Steve Waton, who considered two particular geometric grid classes in his thesis~\cite{waton:on-permutation-:}.  These are the permutations which can be drawn from a circle, later studied by Vatter and Waton~\cite{vatter:on-points-drawn:}, and the permutations which can be drawn on an $\textsf{X}$, later studied by Elizalde~\cite{elizalde:the-x-class-and:}.  Examples of these two grid classes are shown in Figure~\ref{fig-circle-and-X}%
\footnote{The drawing on the left of Figure~\ref{fig-circle-and-X} has been drawn as a diamond to fit with the general approach of this paper, but by stretching and shrinking the $x$- and $y$-axes in the manner formalised in Section~\ref{sec-geometric-perspective}, it is clear that any permutation that can be ``drawn on a diamond'' can also be ``drawn on a circle''.}.

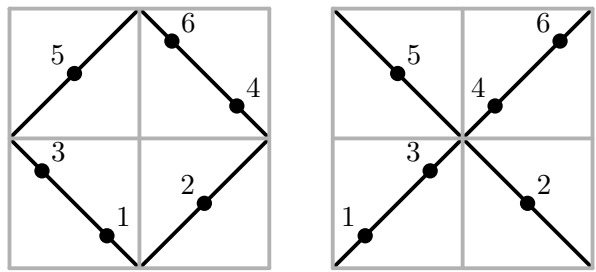
\begin{figure}
\begin{center}
\begin{tabular}{ccc}
\psset{xunit=0.017in, yunit=0.017in}
\psset{linewidth=0.005in}
\begin{pspicture}(0,0)(80,85)
\psline[linecolor=black,linestyle=solid,linewidth=0.02in](0,40)(40,0)
\psline[linecolor=black,linestyle=solid,linewidth=0.02in](0,40)(40,80)
\psline[linecolor=black,linestyle=solid,linewidth=0.02in](40,0)(80,40)
\psline[linecolor=black,linestyle=solid,linewidth=0.02in](40,80)(80,40)
\psline[linecolor=darkgray,linestyle=solid,linewidth=0.02in]{c-c}(0,0)(0,80)
\psline[linecolor=darkgray,linestyle=solid,linewidth=0.02in]{c-c}(40,0)(40,80)
\psline[linecolor=darkgray,linestyle=solid,linewidth=0.02in]{c-c}(80,0)(80,80)
\psline[linecolor=darkgray,linestyle=solid,linewidth=0.02in]{c-c}(0,0)(80,0)
\psline[linecolor=darkgray,linestyle=solid,linewidth=0.02in]{c-c}(0,40)(80,40)
\psline[linecolor=darkgray,linestyle=solid,linewidth=0.02in]{c-c}(0,80)(80,80)
\pscircle*(10,30){0.04in}\uput[45](10,30){$3$}
\pscircle*(20,60){0.04in}\uput[135](20,60){$5$}
\pscircle*(30,10){0.04in}\uput[45](30,10){$1$}
\pscircle*(50,70){0.04in}\uput[45](50,70){$6$}
\pscircle*(60,20){0.04in}\uput[135](60,20){$2$}
\pscircle*(70,50){0.04in}\uput[45](70,50){$4$}
\end{pspicture}
&\rule{0in}{0pt}&
\psset{xunit=0.017in, yunit=0.017in}
\psset{linewidth=0.005in}
\begin{pspicture}(0,0)(80,85)
\psline[linecolor=black,linestyle=solid,linewidth=0.02in](0,0)(40,40)
\psline[linecolor=black,linestyle=solid,linewidth=0.02in](0,80)(40,40)
\psline[linecolor=black,linestyle=solid,linewidth=0.02in](40,40)(80,0)
\psline[linecolor=black,linestyle=solid,linewidth=0.02in](40,40)(80,80)
\psline[linecolor=darkgray,linestyle=solid,linewidth=0.02in]{c-c}(0,0)(0,80)
\psline[linecolor=darkgray,linestyle=solid,linewidth=0.02in]{c-c}(40,0)(40,80)
\psline[linecolor=darkgray,linestyle=solid,linewidth=0.02in]{c-c}(80,0)(80,80)
\psline[linecolor=darkgray,linestyle=solid,linewidth=0.02in]{c-c}(0,0)(80,0)
\psline[linecolor=darkgray,linestyle=solid,linewidth=0.02in]{c-c}(0,40)(80,40)
\psline[linecolor=darkgray,linestyle=solid,linewidth=0.02in]{c-c}(0,80)(80,80)
\pscircle*(10,10){0.04in}\uput[135](10,10){$1$}
\pscircle*(20,60){0.04in}\uput[45](20,60){$5$}
\pscircle*(30,30){0.04in}\uput[135](30,30){$3$}
\pscircle*(50,50){0.04in}\uput[135](50,50){$4$}
\pscircle*(60,20){0.04in}\uput[45](60,20){$2$}
\pscircle*(70,70){0.04in}\uput[135](70,70){$6$}
\end{pspicture}
\end{tabular}
\end{center}
\caption[Two geometric grid classes.]{The permutation $351624$ on the left and the permutation $153426$ on the right lie, respectively, in the geometric grid classes of
$$
\fnmatrix{rr}{1&-1\\-1&1}\mbox{ and }\fnmatrix{rr}{-1&1\\1&-1}.
$$}
\label{fig-circle-and-X}
\end{figure}

A permutation class is said to be \emph{geometrically griddable} if it is contained in some geometric grid class.  With that final piece of terminology, we can state the main results of this paper.
\begin{itemize}
\item {\bf Theorem~\ref{thm-geom-griddable-pwo}.} \emph{Every geometrically griddable class is partially well ordered.}  (Such classes do not contain infinite antichains.)
\item {\bf Theorem~\ref{thm-geom-griddable-fin-basis}.} \emph{Every geometrically griddable class is finitely based.}  (These classes can be defined by only finitely many forbidden patterns.)
\item {\bf Theorem~\ref{thm-geom-griddable-reg-lang}.} \emph{Every geometrically griddable class is in bijection with a regular language, and thus has a rational generating function.}
\item {\bf Theorem~\ref{thm-geom-simple-reg-lang}.} \emph{The simple, sum indecomposable, and skew indecomposable permutations in every geometrically griddable class are each in bijection with a regular language, and thus have rational generating functions.}
\item {\bf Theorem~\ref{thm-geom-griddable-union-atomics}.} \emph{The atomic geometrically griddable classes are precisely the geometric grid classes of $\bullet$-isolated $\zdpm$ matrices, and every geometrically griddable class can be expressed as a finite union of such classes.}  (This type of geometric grid class is defined in Section~\ref{sec-atomic-decompositions}.)
\end{itemize}

For the remainder of the introduction we present the (standard) definitions of permutation classes.  In the next section we formalise the geometric notions of this paper.  Section~\ref{sec-forests} contains a brief discussion of grid classes of forests, while Section~\ref{sec-grid-pmm} introduces partial multiplication matrices.  Sections~\ref{sec-words-and-encodings}--\ref{sec-atomic-decompositions} introduce a correspondence between permutations in a geometric grid class and words, and utilise this correspondence to establish the main results of the paper.  Finally, in Section~\ref{sec-concluding-remarks}, we conclude with numerous open problems.

The permutation $\pi$ of $\{1,2,\dots,n\}$ \emph{contains} or \emph{involves} the permutation $\sigma$ of $\{1,2,\dots,k\}$ (written $\sigma\le\pi$) if $\pi$ has a subsequence of length $k$ which is order isomorphic to $\sigma$.  For example, $\pi=391867452$ (written in list, or one-line notation) contains $\sigma=51342$, as can be seen by considering the subsequence $\pi(2)\pi(3)\pi(5)\pi(6)\pi(9)=91672$.  A \emph{permutation class} is a downset of permutations under this containment ordering; thus if $\C$ is a permutation class, $\pi\in\C$, and $\sigma\le\pi$, then $\sigma\in\C$.  

For any permutation class $\C$ there is a unique (and possibly infinite) antichain $B$ such that
$$
\C=\Av(B)=\{\pi: \beta\not\le\pi\mbox{ for all } \beta \in B\}.
$$
This antichain $B$ is called the \emph{basis} of $\C$.  We denote by $\C_n$ (for $n \in \mathbb{N}$) the set of permutations in $\C$ of length $n$, and we refer to
$$
\sum_{n=0}^\infty |\C_n|x^n=\sum_{\pi\in\C} x^{|\pi|}
$$
as the \emph{generating function} of $\C$.

Finally, a permutation class, or indeed any partially ordered set, is said to be \emph{partially well ordered (pwo)} if it contains neither an infinite strictly descending chain nor an infinite antichain.  Of course, permutation classes cannot contain infinite strictly descending chains, so in this context being pwo is equivalent to having no infinite antichain.

\section{The Geometric Perspective}
\label{sec-geometric-perspective}

Geometric ideas play a significant role in this paper, and to start preparing the ground we reintroduce permutations and the involvement relation in a somewhat nonstandard manner.

We call a subset of the plane a \emph{figure}.  We say that the figure $\mathcal{F}\subseteq\mathbb{R}^2$ is \emph{involved} in the figure $\mathcal{G}$, denoted $\mathcal{F}\le\mathcal{G}$, if there are subsets $A,B\subseteq\mathbb{R}$ and increasing injections $\opinj_x:A\to\mathbb{R}$ and $\opinj_y:B\to\mathbb{R}$ such that
$$
\mathcal{F}\subseteq A\times B
\mbox{ and }
\opinj(\mathcal{F})
=
\{(\opinj_x(a), \opinj_y(b)) \st (a,b)\in\mathcal{F}\}
\subseteq
\mathcal{G}.
$$
The involvement relation is a \emph{preorder} (it is reflexive and transitive but not necessarily antisymmetric) on the collection of all figures.  If $\mathcal{F}\le\mathcal{G}$ and $\mathcal{G}\le\mathcal{F}$ then we say that $\mathcal{F}$ and $\mathcal{G}$ are \emph{(figure) equivalent} and write $\mathcal{F}\equivfig\mathcal{G}$.  Note that in the case of figures with only finitely many points, two figures are equivalent if and only if one can be transformed to the other by stretching and shrinking the axes.  Three concrete examples are provided below.
\begin{itemize}
\item The figures $\mathcal{F}=\{(a,|a|)\st a\in\mathbb{R}\}$ and $\mathcal{G}=\{(a,a^2)\st a\in\mathbb{R}\}$ are equivalent.  To see that $\mathcal{F}\le\mathcal{G}$, take $A\times B=\mathbb{R}\times\mathbb{R}_{\ge 0}$, $\opinj_x(a)=a$, and $\opinj_y(b)=b^2$.  To see that $\mathcal{G}\le\mathcal{F}$, take $A\times B=\mathbb{R}\times\mathbb{R}_{\ge 0}$, $\opinj_x(a)=a$, and $\opinj_y(b)=\sqrt{b}$.
\item The figures $\mathcal{F}=\{(a,a)\st a\in\mathbb{R}\}$ and $\mathcal{G}=\{(a,a)\st a\in[0,1]\cup[2,3]\}$ are equivalent.  To see that $\mathcal{F}\le\mathcal{G}$, take $A\times B=\mathbb{R}^2$, $\opinj_x(a)=1/(1+e^{-a})$, and $\opinj_y(b)=1/(1+e^{-b})$.  To see that $\mathcal{G}\le\mathcal{F}$, consider the identity map with $A\times B=\left([0,1]\cup[2,3]\right)^2$.
\item The ``unit diamond'' defined by $\mathcal{F}=\{(a,b)\st |a|+|b|=1\}$ (shown on the left of Figure~\ref{fig-circle-and-X}) is equivalent to the unit circle $\mathcal{G}=\{(a,b)\st a^2+b^2=1\}$.  To see that $\mathcal{F}\le\mathcal{G}$, take $A\times B=[-1,1]^2$, $\opinj_x(a)=\sin(\pi a/2)$, and $\opinj_y(b)=\cos(\pi(1-b)/2)$.  To see that $\mathcal{G}\le\mathcal{F}$, simply consider the inverses of these maps.
\end{itemize}

To any permutation $\pi$ of length $n$, we associate a figure which we call its \emph{plot}, $\{(i, \pi(i))\}$.  These figures have the important property that they are \emph{independent}, by which we mean that no two points lie on a common horizontal or vertical line.  It is clear that every finite independent figure is equivalent to the plot of a unique permutation, so we could \emph{define} a permutation as an equivalence class of finite independent figures.  Under this identification, the partial order on equivalence classes of finite independent figures is the same as the containment, or involvement, order defined in the introduction.

Every figure $\mathcal{F}\subseteq\mathbb{R}^2$ therefore naturally defines a permutation class,
$$
\Sub(\mathcal{F})
=
\{
\mbox{permutations $\pi$}
\st
\mbox{$\pi$ is equivalent to a finite independent subset of $\mathcal{F}$}
\}.
$$
For example:
\begin{itemize}
\item Let $\mathcal{F} = \{ (x,x) \st x \in \mathbb{R} \}$. Then $\Sub(\mathcal{F})$ contains a single permutation of each length, namely the identity, and its basis is $\{21\}$.
\item Let $\mathcal{F} = \{ (x, x) \st x \in [0,1] \} \cup \{ (x+1, x) \st x\in [0,1] \}$. Then $\Sub(\mathcal{F})$ consists of all permutations having at most one descent.  Its basis is $\{321, 2143, 3142\}$ as can be seen by considering the ways in which two descents might occur.
\item Let $\mathcal{F} = \{ (x, \sin(x) ) \st x \in \mathbb{R} \}$. Then $\Sub(\mathcal{F})$ is the set of all permutations; to establish this note that every permutations can be broken into its increasing contiguous segments (``runs''), points corresponding to each run can be chosen from increasing segments of $\mathcal{F}$.  Its basis is, of course, the empty set.
\end{itemize}
A geometric grid class (as defined in the introduction) is precisely $\Sub(\Lambda)$, where $\Lambda$ denotes the standard figure of the defining matrix.  Permutation classes of the form $\Sub(\mathcal{F})$ in the special case where $\mathcal{F}$ is the plot of a bijection between two subsets of the real numbers have received some study before; we refer the reader to Atkinson, Murphy, and Ru\v{s}kuc~\cite{atkinson:pattern-avoidan:} and Huczynska and Ru\v{s}kuc~\cite{huczynska:pattern-classes:}.

The lines $\{x=k\st k=0,\dots,t\}$ and $\{y=\ell\st\ell=0,\dots,u\}$ play a special role for standard figures, as they divide the figure into its cells.  We extend this notion of \emph{griddings} to all figures.  First though, we need to make a technical observation: because the real number line is order isomorphic to any open interval, it follows that every figure is equivalent to a bounded figure, and thus we may restrict our attention to bounded figures.

Let $R=[a,b]\times [c,d]$ be a rectangle in $\mathbb{R}^2$.  A $t\times u$-gridding of $R$ is a tuple
$$
G=(g_0,g_1,\ldots,g_t; h_0,h_1,\ldots,h_u)
$$
of real numbers satisfying
\begin{eqnarray*}
&a = g_0 \le g_1\le \cdots \le g_t=b,&\\
&c = h_0 \le h_1 \le \cdots \le h_u = d.&
\end{eqnarray*}
These numbers are identified with the corresponding set of vertical and horizontal lines partitioning $R$ into a rectangular collection of cells $C_{k,\ell}$ as shown in Figure \ref{fig-example-gridding}.  We often identify the gridding $G$ and the collection of cells $C_{k,\ell}$.

\begin{figure}
\begin{center}
\psset{xunit=0.017in, yunit=0.017in}
\psset{linewidth=0.005in}
\begin{pspicture}(-5,-5)(120,80)
\psline[linecolor=black,linestyle=dashed,linewidth=0.01in]{c-c}(30,0)(30,20)
\psline[linecolor=black,linestyle=dashed,linewidth=0.01in]{c-c}(110,0)(110,20)
\psline[linecolor=black,linestyle=dashed,linewidth=0.01in]{c-c}(0,20)(30,20)
\psline[linecolor=black,linestyle=dashed,linewidth=0.01in]{c-c}(0,70)(30,70)
\psline[linecolor=black,linestyle=solid,linewidth=0.02in]{c-c}(30,20)(110,20)(110,70)(30,70)(30,20)
\psline[linecolor=darkgray,linestyle=solid,linewidth=0.02in]{c-c}(50,0)(50,80)
\psline[linecolor=darkgray,linestyle=solid,linewidth=0.02in]{c-c}(85,0)(85,80)
\psline[linecolor=darkgray,linestyle=solid,linewidth=0.02in]{c-c}(0,50)(120,50)
\psaxes[linewidth=0.01in,ticks=none,showorigin=false,labels=none,arrowscale=2]{->}(0,0)(0,0)(120,80)
\rput[c](40,35){$C_{11}$}
\rput[c](67.5,35){$C_{21}$}
\rput[c](97.5,35){$C_{31}$}
\rput[c](40,60){$C_{12}$}
\rput[c](67.5,60){$C_{22}$}
\rput[c](97.5,60){$C_{32}$}
\uput[270](30,0){$a$}
\uput[270](50,0){$g_1$}
\uput[270](85,0){$g_2$}
\uput[270](110,0){$b$}
\uput[180](0,20){$c$}
\uput[180](0,50){$h_1$}
\uput[180](0,70){$d$}
\end{pspicture}
\end{center}
\caption{A $3\times 2$ gridding of the rectangle $[a,b]\times[c,d]$ with the cells $C_{k,\ell}$ indicated.}
\label{fig-example-gridding}
\end{figure}
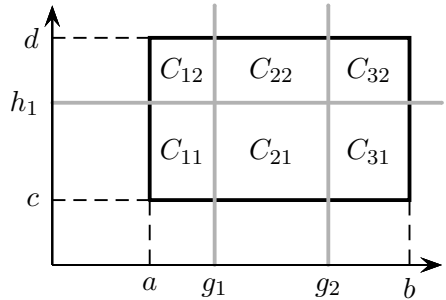

A \emph{gridded figure} is a pair $(\mathcal{F},G)$ where $\mathcal{F}$ is a figure and $G$ is a gridding containing $\mathcal{F}$ in its interior.  To ensure that each point of $\mathcal{F}$ lies in a unique cell, we also require that the grid lines are disjoint from $\mathcal{F}$. By analogy with ungridded figures and permutations, we define the preorder $\le$ and the equivalence relation $\equivfig$ for $t\times u$ gridded figures.  The additional requirement is that the mapping $\phi=(\phi_x,\phi_y)$ appearing in the original definition maps the $(k,\ell)$ cell of one gridded figure to the $(k,\ell)$ cell of the other gridded figure.  A \emph{gridded permutation} is the equivalence class of a finite independent gridded figure.

The connection between finite figures, gridded finite figures, permutations, and gridded permutations can be formalised as follows.  Let $\Phi$ and $\Phi^\gridded$, respectively, denote the set of all finite independent and finite gridded independent figures in $\mathbb{R}^2$, and let $\mathcal{S}$ and $\mathcal{S}^\gridded$, respectively, denote the set of all permutations and all gridded permutations.  The (obvious) mappings connecting these sets are:
\begin{itemize}
\item $\Delta\st \Phi^\gridded\rightarrow \Phi$, given by removing the grid lines, i.e., $(\mathcal{F},G)\mapsto \mathcal{F}$;
\item $\delta\st\mathcal{S}^\gridded\rightarrow\mathcal{S}$, also given by removing the grid lines, i.e., in this context, mapping the equivalence class of $(\mathcal{F},G)$ under $\equivfig$ to the equivalence class of $\mathcal{F}$ under $\equivfig$;
\item $\pi^\gridded\st \Phi^\gridded\rightarrow \mathcal{S}^\gridded$, which sends every element of $\Phi^\gridded$ to its equivalence class under $\equivfig$;
\item $\pi\st \Phi\rightarrow \mathcal{S}$, which sends each element of $\Phi$ to its equivalence class under $\equivfig$.
\end{itemize}
It is a routine matter to verify that the diagram
$$
\psmatrix[colsep=1cm,rowsep=1cm]
\Phi^\gridded & \Phi \\
\mathcal{S}^\gridded & \mathcal{S}
\endpsmatrix
\psset{nodesep=3pt,arrows=->}
\ncline{1,1}{1,2}
\taput{\Delta}
\ncline{2,1}{2,2}
\tbput{\delta}
\ncline{1,1}{2,1}
\tlput{\pi^\gridded}
\ncline{1,2}{2,2}
\trput{\pi}
$$
commutes.

The griddings we are interested in restrict the content of cells in a way specified by a matrix.  Given a $\zpm$ matrix $M$, we say that the gridded figure $(\mathcal{F},G) $ is \emph{compatible} with $M$ if the following holds for all relevant $k$ and $\ell$:
$$
\mathcal{F}\cap C_{k,\ell} \mbox{ is } \left\{
\begin{array}{ll}
\mbox{increasing,} & \mbox{if } M_{k,\ell}=1,\\
\mbox{decreasing,} & \mbox{if } M_{k,\ell}=-1,\\
\mbox{empty}, & \mbox{if } M_{k,\ell}=0.
\end{array}
\right.
$$
We are concerned with the images, under the maps $\pi^\gridded$ and $\pi\Delta=\delta\pi^\gridded$, of the set of all $M$-compatible, finite, independent, gridded figures; we denote these two sets by $\Grid^\gridded(M)$ and $\Grid(M)$, respectively.  We refer to $\Grid(M)$ as the \emph{(monotone) grid class} of $M$.

In this context, the standard figure $\Lambda$ (as defined in the introduction) of a $\zpm$ matrix $M$ of size $t\times u$ has \emph{standard grid lines} $G$ given by $\{x=k\st k=0,\dots,t\}$ and $\{y=\ell\st \ell=0,\dots,u\}$.  We refer to the pair $(\Lambda,G)$ as the \emph{standard gridded figure} of $M$, and denote it by $\Lambda^\gridded$; note that $\Lambda^\gridded$ is compatible (in the sense above) with $M$.  The set of all gridded permutations contained in $\Lambda^\gridded$ is denoted by $\Geom^\gridded(M)$, and its image under $\delta$ is $\Geom(M)$, as defined in the introduction.  Because the standard gridded figure of a $\zpm$ matrix $M$ is compatible with $M$, we always have
$$
\Geom(M)\subseteq\Grid(M).
$$
In the next section we characterise the matrices for which equality is achieved.

\section{Grid Classes of Forests}
\label{sec-forests}

The first appearance of grid classes in the literature, in the special case where $M$ is a permutation matrix and under the name ``profile classes'', was in Atkinson~\cite{atkinson:restricted-perm:}.  Atkinson observed that such grid classes have polynomial enumeration.  Huczynska and Vatter~\cite{huczynska:grid-classes-an:}, who introduced the name ``grid classes'', generalised Atkinson's result to show that grid classes of signed permutation matrices have eventually polynomial enumeration.  This result allowed them to giver a more structural proof of Kaiser and Klazar's ``Fibonacci dichotomy''~\cite{kaiser:on-growth-rates:}, which states that for every permutation class $\C$, either $|\C_n|$ is greater than the $n$th Fibonacci number for all $n$, or $|\C_n|$ is eventually polynomial.

Atkinson, Murphy, and Ru\v{s}kuc~\cite{atkinson:partially-well-:} and Albert, Atkinson, and Ru\v{s}kuc~\cite{albert:regular-closed-:} studied the special case of grid classes of $\zpm$ vectors, under the name ``$W$-classes''.  The former paper proves that such grid classes are pwo and finitely based, while the latter shows that they have rational generating functions.

The first appearance of grid classes of $\zpm$ matrices in full generality was in Murphy and Vatter~\cite{murphy:profile-classes:} (again under the name ``profile classes'').  They were interested in the pwo properties of such classes, and to state their result we need a definition.


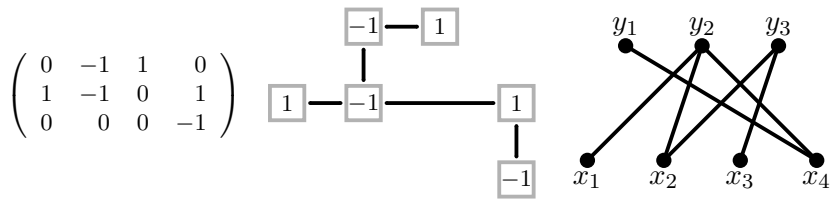
\begin{figure}[t]
\begin{center}
\begin{tabular}{ccccc}
\psset{xunit=0.02in, yunit=0.02in}
\psset{linewidth=0.005in}
\begin{pspicture}(0,0)(52,54)
\rput[l](0,27){%
	$\fnmatrix{rrrr}{0&-1&1&0\\1&-1&0&1\\0&0&0&-1}$
}
\end{pspicture}
&\rule{0in}{0pt}&
\psset{xunit=0.02in, yunit=0.02in}%
\psset{linewidth=0.005in}%
\begin{footnotesize}\begin{pspicture}(0,0)(67,54)
\psframe[linecolor=darkgray,linewidth=0.02in](0,20)(10,30)
\rput[c](5,25){$1$}
\psframe[linecolor=darkgray,linewidth=0.02in](20,20)(30,30)
\rput[c](25,25){$-1$}
\psframe[linecolor=darkgray,linewidth=0.02in](20,40)(30,50)
\rput[c](25,45){$-1$}
\psframe[linecolor=darkgray,linewidth=0.02in](40,40)(50,50)
\rput[c](45,45){$1$}
\psframe[linecolor=darkgray,linewidth=0.02in](60,20)(70,30)
\rput[c](65,25){$1$}
\psframe[linecolor=darkgray,linewidth=0.02in](60,0)(70,10)
\rput[c](65,5){$-1$}
\psline[linecolor=black,linewidth=0.02in]{c-c}(11,25)(19,25)
\psline[linecolor=black,linewidth=0.02in]{c-c}(31,25)(59,25)
\psline[linecolor=black,linewidth=0.02in]{c-c}(25,31)(25,39)
\psline[linecolor=black,linewidth=0.02in]{c-c}(31,45)(39,45)
\psline[linecolor=black,linewidth=0.02in]{c-c}(65,19)(65,11)
\end{pspicture}\end{footnotesize}
&\rule{0in}{0pt}&
\psset{xunit=0.02in, yunit=0.02in}
\begin{pspicture}(5,0)(70,50)
\pscircle*(5,10){0.04in}
\pscircle*(25,10){0.04in}
\pscircle*(45,10){0.04in}
\pscircle*(65,10){0.04in}
\pscircle*(15,40){0.04in}
\pscircle*(35,40){0.04in}
\pscircle*(55,40){0.04in}
\rput[c](5,5){$x_1$}
\rput[c](25,5){$x_2$}
\rput[c](45,5){$x_3$}
\rput[c](65,5){$x_4$}
\rput[c](15,45){$y_1$}
\rput[c](35,45){$y_2$}
\rput[c](55,45){$y_3$}
\psline[linecolor=black,linewidth=0.02in](5,10)(35,40) 
\psline[linecolor=black,linewidth=0.02in](25,10)(35,40) 
\psline[linecolor=black,linewidth=0.02in](25,10)(55,40) 
\psline[linecolor=black,linewidth=0.02in](45,10)(55,40) 
\psline[linecolor=black,linewidth=0.02in](65,10)(15,40) 
\psline[linecolor=black,linewidth=0.02in](65,10)(35,40) 
\end{pspicture}
\end{tabular}
\end{center}
\caption{A matrix together with its cell graph (centre) and row-column graph (right).}
\label{fig-graph-grid}
\end{figure}

The \emph{cell graph} of $M$ is the graph on the vertices $\{(k,\ell) : M_{k,\ell}\neq 0\}$ in which $(k,\ell)$ and $(i,j)$ are adjacent if the corresponding cells of $M$ share a row or a column and there are no nonzero entries between them in this row or column.  We say that the matrix $M$ is a \emph{forest} if its cell graph is a forest.  Viewing the absolute value of $M$ as the adjacency matrix of a bipartite graph, we obtain a different graph, its \emph{row-column graph}%
\footnote{In other words, the row-column graph of a $t\times u$ matrix $M$ is the bipartite graph on the vertices $x_1,\dots,x_t,y_1,\dots,y_u$ where there is an edge between $x_k$ and $y_\ell$ if and only if $M_{k,\ell}\neq 0$.}.
It is not difficult to show that the cell graph of a matrix is a forest if and only if its row-column graph is also a forest.  An example of each of these graphs is shown in Figure~\ref{fig-graph-grid}.  These graphs completely determine the pwo properties of grid classes:

\begin{theorem}[Murphy and Vatter~\cite{murphy:profile-classes:}, later generalised by Brignall~\cite{brignall:grid-classes-an:}]
\label{thm-forest-grids-pwo}
The class $\Grid(M)$ is pwo if and only if $M$ is a forest.
\end{theorem}

It has long been conjectured that $\Grid(M)$ has a rational generating function if $M$ is a forest%
\footnote{Note that $\Grid(M)$ can have a nonrational generating function when $M$ is not a forest.  An example of this is given in the conclusion.};
for example, by Huczynska and Vatter~\cite[Conjecture 2.8]{huczynska:grid-classes-an:}.  Indeed, this conjecture was the original impetus for the present work.  However, as the work progressed, it became apparent that the geometric paradigm provided a viewpoint which was at once more insightful and more general, and thus our perspective shifted.  The link with the original motivation is provided by the following result.

\begin{theorem}
\label{thm-forests-are-geoms}
If $M$ is a forest then $\Grid^\gridded(M)=\Geom^\gridded(M)$, and thus $\Grid(M)=\Geom(M)$.
\end{theorem}
\begin{proof}
The proof is by induction on the number of nonzero entries of $M$.  For the case of a single nonzero cell, note that one can place any increasing (resp., decreasing) set of points on a line of slope $1$ (resp., $-1$) by applying a horizontal transformation.

Now suppose that $M$ has two or more nonzero entries, denote its standard gridded figure by $\Lambda_M^\gridded=(\Lambda_M,G)$, and let $(k,\ell)$ denote a leaf in the cell graph of $M$.  By considering the transpose of $M$ if necessary, we may assume that there are no other nonzero entries in column $k$ of $M$.  Let $\pi^\gridded$ be an arbitrary gridded permutation in $\Grid^\gridded(M)$.  We aim to show that $\pi^\gridded\in\Geom^\gridded(M)$.

Denote by $N$ the matrix obtained from $M$ by setting the $(k,\ell)$ entry equal to $0$, and denote its standard gridded figure by $\Lambda_N^\gridded=(\Lambda_N,G)$; note that the grid lines of $\Lambda_N^\gridded$ and $\Lambda_M^\gridded$ are identical because the corresponding matrices are the same size.  Let $\sigma^\gridded$ denote the gridded permutation obtained from $\pi^\gridded$ by removing all entries in the $(k,\ell)$ cell.  Because $\sigma^\gridded$ lies in $\Grid^\gridded(N)$, which by induction is equal to $\Geom^\gridded(N)$, there is a finite independent point set $S\subseteq\Lambda_N\subseteq\Lambda_M$ such that $(S,G)\equivfig \sigma^\gridded$.  So long as we do not demand that the new points belong to $\Lambda_M$, it is clear that we can extend $S$ by adding points in the $(k,\ell)$ cell to arrive at a point set $P\supseteq S$ such that $(P,G)\equivfig\pi^\gridded$.  Then we can apply a horizontal transformation to column $k$ to move these new points onto the diagonal line segment of $\Lambda_M$ in this cell.  This horizontal transformation does not affect the points of $S$, because none of those points lie in column $k$, so we see that $P\subseteq\Lambda_M$, and thus $\pi^\gridded\in\Geom^\gridded(M)$, as desired.
\end{proof}

\begin{figure}
\begin{center}
\psset{xunit=0.0226666666666in, yunit=0.0226666666666in}
\psset{linewidth=0.005in}
\begin{pspicture}(0,0)(60,63.75)
\psline[linecolor=black,linestyle=dashed,linewidth=0.01in]{c-c}(10.85,20)(55,20)
\psline[linecolor=black,linestyle=dashed,linewidth=0.01in]{c-c}(20,49.15)(20,5)
\psline[linecolor=black,linestyle=dashed,linewidth=0.01in]{c-c}(40,10.85)(40,55)
\psline[linecolor=black,linestyle=dashed,linewidth=0.01in]{c-c}(49.15,40)(5,40)
\psline[linecolor=darkgray,linestyle=solid,linewidth=0.02in]{c-c}(0,0)(0,60)
\psline[linecolor=darkgray,linestyle=solid,linewidth=0.02in]{c-c}(30,0)(30,60)
\psline[linecolor=darkgray,linestyle=solid,linewidth=0.02in]{c-c}(60,0)(60,60)
\psline[linecolor=darkgray,linestyle=solid,linewidth=0.02in]{c-c}(0,0)(60,0)
\psline[linecolor=darkgray,linestyle=solid,linewidth=0.02in]{c-c}(0,30)(60,30)
\psline[linecolor=darkgray,linestyle=solid,linewidth=0.02in]{c-c}(0,60)(60,60)
\pscircle*(10,20){0.04in}
\pscircle*(20,50){0.04in}
\pscircle*(40,10){0.04in}
\pscircle*(50,40){0.04in}
\uput[270](10,20){$2$}
\uput[180](20,50){$4$}
\uput[0](40,10){$1$}
\uput[90](50,40){$3$}
\end{pspicture}
\end{center}
\caption[The permutation $2413$ cannot be drawn on an $\textsf{X}$.]{The permutation $2413$ lies in $\Grid\fnmatrix{rr}{-1&1\\1&-1}$ but not $\Geom\fnmatrix{rr}{-1&1\\1&-1}$.  If $2413$ did lie in this geometric grid class, then beginning with the $2$ and moving in a clockwise direction, we see that the $4$ lies to the right of the $2$ and thus closer to the centre, the $3$ lies closer than the $4$ to the centre, the $1$ lies closer than the $3$ to the centre, and finally, to reach a contradiction, the $2$ must lie even closer to the centre than the $1$.}
\label{fig-2413-not-in-X}
\end{figure}

Because of Theorem~\ref{thm-forests-are-geoms}, all of our results about geometric grid classes yield immediate corollaries to grid classes of forests, which we shall generally not mention.  For example, our upcoming Theorem~\ref{thm-geom-griddable-pwo} generalises one direction of Theorem~\ref{thm-forest-grids-pwo}.

The fact that $\Grid(M)$ is not pwo when $M$ is not a forest (the other direction of Theorem~\ref{thm-forest-grids-pwo}), combined with Theorem~\ref{thm-geom-griddable-pwo} which shows that all geometric grid classes are pwo, implies that the converse to Theorem~\ref{thm-forests-are-geoms} also holds: if $\Grid(M)=\Geom(M)$ then $M$ is a forest.  This fact can also be established by arguments generalising those accompanying Figure~\ref{fig-2413-not-in-X}.

\section{Partial Multiplication Matrices}
\label{sec-grid-pmm}

In this section we consider a particular ``refinement'' operation on matrices, which is central to our later arguments.  Let $M$ be a $\zpm$ matrix of size $t\times u$, and $q$ a positive integer.  The \emph{refinement} $M^{\times q}$ of $M$ is the $\zpm$ matrix of size $qt\times qu$ obtained from $M$ by replacing each $1$ by a $q\times q$ identity matrix (which, by our conventions, has ones along its southwest-northeast diagonal), each $-1$ by a negative $q\times q$ anti-identity matrix, and each $0$ by a $q\times q$ zero matrix.  It is easy to see that the standard figure $M^{\times q}$ is equivalent to the standard figure of $M$, so $\Geom(M^{\times q})=\Geom(M)$ for all $q$ (although, of course, the corresponding gridded classes differ).  

The refinements $M^{\times 2}$ play a special role throughout this paper.  To explain this we first need a definition.  We say that a $\zpm$ matrix $M$ of size $t\times u$ is a \emph{partial multiplication matrix} if there are \emph{column and row signs}
$$
c_1,\ldots,c_t,r_1,\ldots,r_u\in \{1,-1\}
$$
such that every entry $M_{k,\ell}$ is equal to either $0$ or the product $c_kr_\ell$.

%
%

As our next result shows, we are never far from a partial multiplication matrix.

\begin{proposition}
For every $\zpm$ matrix $M$, its refinement $M^{\times 2}$ is a partial multiplication matrix.
\end{proposition}
\begin{proof}
By construction, $M^{\times 2}$ is made up of $2\times 2$ blocks equal to
$$
\fnmatrix{rr}{0&0\\0&0},
\fnmatrix{rr}{0&1\\1&0},
\mbox{ and }
\fnmatrix{rr}{-1&0\\0&-1}.
$$
From this it follows that $(M^{\times 2})_{k,\ell}\in\{0,(-1)^{k+\ell}\}$.  Therefore we may take $c_k=(-1)^k$ and $r_\ell=(-1)^\ell$ as our column and row signs.
\end{proof}

Because $\Geom(M)=\Geom(M^{\times 2})$, we may always assume that the matrices we work with are partial multiplication matrices.  We record this useful fact below.

\begin{proposition}
\label{prop-geom-pmm}
Every geometric grid class is the geometric grid class of a partial multiplication matrix.
\end{proposition}

\section{Words and Encodings}
\label{sec-words-and-encodings}

From the point of view of our goals in this paper, subword-closed languages over a finite alphabet display model behaviour: all such languages are defined by finite sets of forbidden subwords, are pwo under the subword order, and have rational generating functions.  The brunt of our subsequent effort is focused on transferring these favourable properties from words to permutations.

Let $\Sigma$ be a finite alphabet, and $\Sigma^\ast$ the set of all finite words (i.e., sequences) over $\Sigma$. This set is partially ordered by means of the \emph{subword} or \emph{subsequence} order: $v\le w$ if one can obtain $v$ from $w$ by deleting letters.

Subsets of $\Sigma^\ast$ are called \emph{languages}.  We say that a language is \emph{subword-closed} if it is a downward closed set in the subword order (such languages are also called \emph{piecewise testable} by some, for example, Simon~\cite{simon:piecewise-testa:}).  To borrow terminology from permutation classes, we say that the \emph{basis} of a subword-closed language $L$ is the set of minimal words which do not lie in $L$.  It follows that $L$ consists of precisely those words which do not contain any element of its basis.  Moreover, a special case of a result of Higman~\cite{higman:ordering-by-div:} implies that subword-closed languages have finite bases:

\newtheorem*{higmans-theorem}{\rm\bf Higman's Theorem}
\begin{higmans-theorem}
\emph{The set of words over any finite alphabet is pwo under the subword order.}
\end{higmans-theorem}

The following characterisation of subword-closed languages is folkloric, and follows directly from the fact that, for any finite set of forbidden subwords, there exists a finite state automaton accepting words not containing these subwords, which is acyclic except for loops at individual states.

\begin{proposition}\label{prop-subword-closed-decomp}
Let $\Sigma$ be a finite alphabet.  Every non-empty subword-closed language over $\Sigma$ can be expressed as a finite union of languages of the form
$$
\Sigma_1^\ast \{\emptyword, a_2\} \Sigma_3^\ast\{\emptyword, a_4\} \ldots \Sigma_{2q}^\ast \{\emptyword, a_{2q}\} \Sigma_{2q+1}^\ast
$$
where $q\geq 0$, $a_2,\ldots ,a_{2q}\in \Sigma$ and $\Sigma_1,\ldots,\Sigma_{2q+1}\subseteq \Sigma$.
\end{proposition}

This fact shows that subword-closed languages are \emph{regular languages}.  To recall their definition briefly, given a finite alphabet $\Sigma$, the empty language $\emptyset$, the empty word language $\{\emptyword\}$, and the singleton languages $\{a\}$ for each $a\in\Sigma$ are regular; moreover, given two regular languages $K$ and $L$ over $\Sigma$, their union $K\cup L$, their concatenation $KL=\{vw\st v\in K\mbox{ and }w\in L\}$, and the star $L^\ast=\{v^{(1)}\cdots v^{(m)}\st m\ge 0\mbox{ and }v^{(1)},\dots,v^{(m)}\in L\}$ are also regular.  Alternatively, one may define regular languages as those accepted by a deterministic finite state automaton.  From this second viewpoint, it is easy to see that given two regular languages $K$ and $L$, their complement $K\setminus L$ is also regular, a property we use many times.  We say that the \emph{generating function} of the language $L$ is $\sum x^{|w|}$ where the sum is taken over all $w\in L$ and $|w|$ denotes the number of letters in $w$, i.e., its \emph{length}.  We use only the most basic properties of regular languages, for which we refer the reader to Flajolet and Sedgewick~\cite[Section I.4 and Appendix A.7]{flajolet:analytic-combin:}.  In particular, the following fact is of central importance throughout the paper.

\begin{theorem}\label{thm-reg-lang-rat-gf}
Every regular language has a rational generating function.
\end{theorem}

We now describe the correspondence between permutations in a geometric grid class, $\Geom(M)$, and words over an appropriate finite alphabet.  This encoding, essentially introduced by Vatter and Waton~\cite{vatter:on-partial-well:}, is central to all of our proofs.

By Proposition~\ref{prop-geom-pmm}, we may assume that $M$ is a $t\times u$ partial multiplication matrix with column and row signs $c_1,\dots,c_t$ and $r_1,\dots,r_u$.  Let $\Lambda^\gridded$ denote the standard gridded figure of $M$, and define the \emph{cell alphabet} of $M$ to be
$$
\Sigma=\{\cell{k}{\ell} \st M_{k,\ell}\neq 0\}.
$$
Intuitively, the letter $\cell{k}{\ell}$ represents an instruction to place a point in an appropriate position on the line in the $(k,\ell)$ cell of $\Lambda^\gridded$.  This appropriate position is determined as follows, and the whole process is depicted in Figure~\ref{fig-vatter-waton-map-example}.

\begin{figure}
\begin{center}
\psset{xunit=0.017in, yunit=0.017in}
\psset{linewidth=0.005in}
\begin{pspicture}(-3,-3)(120,80)
\psline[linecolor=black,linestyle=solid,linewidth=0.02in](0,0)(40,40)
\psline[linecolor=black,linestyle=solid,linewidth=0.02in](40,40)(80,0)
\psline[linecolor=black,linestyle=solid,linewidth=0.02in](40,40)(80,80)
\psline[linecolor=black,linestyle=solid,linewidth=0.02in](80,40)(120,0)
\psline[linecolor=black,linestyle=solid,linewidth=0.02in](80,40)(120,80)
\psline[linecolor=darkgray,linestyle=solid,linewidth=0.02in]{c-c}(0,0)(0,80)
\psline[linecolor=darkgray,linestyle=solid,linewidth=0.02in]{c-c}(40,0)(40,80)
\psline[linecolor=darkgray,linestyle=solid,linewidth=0.02in]{c-c}(80,0)(80,80)
\psline[linecolor=darkgray,linestyle=solid,linewidth=0.02in]{c-c}(120,0)(120,80)
\psline[linecolor=darkgray,linestyle=solid,linewidth=0.02in]{c-c}(0,0)(120,0)
\psline[linecolor=darkgray,linestyle=solid,linewidth=0.02in]{c-c}(0,40)(120,40)
\psline[linecolor=darkgray,linestyle=solid,linewidth=0.02in]{c-c}(0,80)(120,80)
\pscircle*(85,35){0.04in}
\uput[10](85,35){$p_1$}
\pscircle*(90,30){0.04in}
\uput[270](90,30){$p_2$}
\pscircle*(55,55){0.04in}
\uput[-45](55,55){$p_3$}
\pscircle*(60,20){0.04in}
\uput[45](60,20){$p_4$}
\pscircle*(15,15){0.04in}
\uput[-45](15,15){$p_5$}
\pscircle*(110,70){0.04in}
\uput[-60](110,70){$p_6$}
\pscircle*(75,75){0.04in}
\uput[275](75,75){$p_7$}
\psline[linecolor=black,linestyle=solid,linewidth=0.01in,arrowsize=0.05in]{<-c}(-3,1)(-3,39)
\psline[linecolor=black,linestyle=solid,linewidth=0.01in,arrowsize=0.05in]{c->}(-3,41)(-3,79)
\psline[linecolor=black,linestyle=solid,linewidth=0.01in,arrowsize=0.05in]{<-c}(1,-3)(39,-3)
\psline[linecolor=black,linestyle=solid,linewidth=0.01in,arrowsize=0.05in]{c->}(41,-3)(79,-3)
\psline[linecolor=black,linestyle=solid,linewidth=0.01in,arrowsize=0.05in]{c->}(81,-3)(119,-3)
\end{pspicture}
\end{center}
\caption[An example of the map $\bij$.]{An example of the map $\bij$ for the matrix $\fnmatrix{rrr}{0&1&1\\1&-1&-1}$ with row signs $r_1=-1$ and $r_2=1$ and column signs $c_1=-1$, $c_2=c_3=1$.  Here we see that $\bij(\cell{3}{1}\cell{3}{1}\cell{2}{2}\cell{2}{1}\cell{1}{1}\cell{3}{2}\cell{2}{2})=1527436$.}
\label{fig-vatter-waton-map-example}
\end{figure}

We say that the \emph{base line} of a column of $\Lambda^\gridded$ is the grid line to the left (resp., right) of that column if the corresponding column sign is $1$ (resp., $-1$).  Similarly, the base line of a row of $\Lambda^\gridded$ is the grid line below (resp., above) that row if the corresponding row sign is $1$ (resp., $-1$).  We designate the intersection of the two base lines of a cell as its \emph{base point}.  Note that the base point is an endpoint of the line segment of $\Lambda$ lying in this cell.  As this definition indicates, we interpret the column and row signs as specifying the direction in which the columns and rows are ``read''.  Owing to this interpretation, we represent the column and row signs in our figures by arrows, as shown in Figure~\ref{fig-vatter-waton-map-example}.

To every word $w=w_1\cdots w_n\in\Sigma^\ast$ we associate a permutation $\bij(w)$ as follows.  First we choose arbitrary distances
$$
0<d_1<\cdots<d_n<1.
$$
Next, for each $i$, we let $p_i$ be the point on the line segment in cell $C_{k,\ell}$, where $w_i=a_{k,\ell}$, at infinity-norm distance $d_i$ from the base point of $C_{k,\ell}$.  It follows from our choice of distances $d_1,\dots,d_n$ that $p_1,\dots,p_n$ are independent, and we define $\bij(w)$ to be the permutation which is equivalent to the set $\{p_1,\dots,p_n\}$ of points.

It is a routine exercise to show that $\bij(w)$ does not depend on the particular choice of distances $d_1,\dots,d_n$, showing that the mapping $\bij\st\Sigma^\ast\to \Geom(M)$ is well-defined.  Of course there is a gridded counterpart $\bij^\gridded\st \Sigma^\ast \to \Geom^\gridded(M)$, whereby we retain the grid lines coming from the figure $\Lambda^\gridded$.

The basic properties of $\bij$ and $\bij^\gridded$ are described by the following result.

\begin{proposition}
\label{prop-properties-of-bij}
The mappings $\bij$ and $\bij^\gridded$ are length-preserving, finite-to-one, onto, and order-preserving.
\end{proposition}

\begin{proof}
That $\bij$ is length-preserving is obvious as it maps letters in a word to entries in a permutation, and that it is finite-to-one follows immediately from this.

In order to prove that $\bij$ is onto, let $\pi\in\Geom(M)$ and choose a finite set $P=\{p_1,\dots,p_n\}\subseteq\Lambda$ of points which represent $\pi$ (where as usual $\Lambda$ denotes the standard figure of $M$).  Suppose that the point $p_i$ belongs to the cell $(k_i,\ell_i)$ of $\Lambda^\gridded$, and let $d_i$ denote the infinity-norm distance from $p_i$ to the base point of this cell.  The points in $P$ are independent, because they are equivalent to a permutation.  Therefore, we may move the points of $P$ independently by small amounts without affecting its (figure) equivalence class, and thus may assume that the distances $d_i$ are distinct.  By reordering the points if necessary, we may also assume that $d_1<\cdots<d_n$.  It is then clear that $\bij(\cell{k_1}{\ell_1}\cdots\cell{k_n}{\ell_n})=\pi$, so $\bij$ is indeed onto.

It remains to show that $\bij$ is order-preserving.  Suppose that $v=v_1\cdots v_k,w=w_1\cdots w_n\in\Sigma^\ast$ satisfy $v\le w$.  Thus there are indices $1\le i_1<\cdots<i_k\le n$ such that $v=w_{i_1}\cdots w_{i_k}$.  Note that if $\bij(w)$ is represented by the point set $\{p_1,\dots,p_n\}$ via the sequence of distances $d_1<\cdots<d_n$, then $\bij(v)$ is represented by the point set $\{p_{i_1},\dots,p_{i_k}\}$ via the sequence of distances $d_{i_1}<\cdots<d_{i_k}$, so $\bij(v)\le\bij(w)$.

The proofs for the gridded version $\bij^\gridded$ are analogous.
\end{proof}

Using this correspondence between words and permutations, one may give an alternative proof of Theorem~\ref{thm-forests-are-geoms}, showing that $\Grid^\gridded(M)=\Geom^\gridded(M)$, and thus that $\Grid(M)=\Geom(M)$ when $M$ is a forest.  As Proposition~\ref{prop-properties-of-bij} shows that $\bij^\gridded$ maps onto $\Geom^\gridded(M)$, one only need to show that it also maps onto $\Grid^\gridded(M)$ when $M$ is a forest.  This is proved directly in Vatter and Waton~\cite{vatter:on-partial-well:}.

\section{Partial Well Order and Finite Bases}
\label{sec-geom-fin-basis}

We now use the encoding $\bij:\Sigma^\ast\rightarrow\Geom(M)$ from the previous section to establish structural properties of geometrically griddable classes, i.e., subclasses of geometric grid classes.  We begin with partial well order.  By Theorem~\ref{thm-forests-are-geoms}, this generalises one direction of Murphy and Vatter's Theorem~\ref{thm-forest-grids-pwo}.

\begin{theorem}
\label{thm-geom-griddable-pwo}
Every geometrically griddable class is partially well ordered.
\end{theorem}
\begin{proof}
Let $\C$ be a geometrically griddable class.  By Proposition~\ref{prop-geom-pmm}, $\C\subseteq\Geom(M)$ for some partial multiplication matrix $M$.  As partial well order is inherited by subclasses, it suffices to prove that $\Geom(M)$ is pwo.  Let $\bij\st\Sigma^\ast\rightarrow\Geom(M)$, where $\Sigma$ is the cell alphabet of $M$, be the encoding introduced in Section~\ref{sec-words-and-encodings}.  Take $A\subseteq\Geom(M)$ to be an antichain.  For every $\alpha\in A$ there is some $w_\alpha\in\Sigma^\ast$ such that $\bij(w_\alpha)=\alpha$ because $\bij$ is onto (Proposition~\ref{prop-properties-of-bij}).  The set $\{w_\alpha\st\alpha\in A\}$ must be an antichain in $\Sigma^\ast$ because $\bij$ is order-preserving (Proposition~\ref{prop-properties-of-bij} again).  Higman's Theorem therefore shows that this set, and thus also $A$, is finite, as desired.
\end{proof}

Our next goal is the following result.

\begin{theorem}
\label{thm-geom-griddable-fin-basis}
Every geometrically griddable class is finitely based.
\end{theorem}

We first make an elementary observation.

\begin{proposition}
\label{prop-geom-griddable-unions}
The union of a finite number of geometrically griddable classes is geometrically griddable.
\end{proposition}
\begin{proof}
Let $\C$ and $\D$ be geometrically griddable classes.  Thus there are matrices $M$ and $N$ such that $\C\subseteq\Geom(M)$ and $\D\subseteq\Geom(N)$.  It follows that $\C\cup\D\subseteq\Geom(P)$ for any matrix $P$ which contains copies of both $M$ and $N$.  The result for arbitrary finite unions follows by iteration.
\end{proof}

Given any permutation class $\C$, we let $\C^{+1}$ denote the class of \emph{one-point extensions} of elements of $\C$, that is, $\C^{+1}$ is the set of all permutations $\pi$ which contain an entry whose removal yields a permutation in $\C$.  Every basis element of a class $\C$ is necessarily a one-point extension of $\C$, because the removal of \emph{any} entry of a basis element of $\C$ yields a permutation in $\C$.  Since bases of permutation classes are necessarily antichains, Theorem~\ref{thm-geom-griddable-fin-basis} will follow from the following result and Theorem~\ref{thm-geom-griddable-pwo}.

\begin{theorem}
\label{thm-geom-griddable-one-point-extension}
If the class $\C$ is geometrically griddable, then the class $\C^{+1}$ is also geometrically griddable.
\end{theorem}
\begin{proof}
It suffices to prove the result for geometric grid classes themselves, and by Proposition~\ref{prop-geom-pmm}, we may further restrict our attention to the case of $\Geom(M)$ where $M$ is a partial multiplication matrix.  Take $\tau\in\Geom(M)^{+1}$ to be a one-point extension of $\pi\in\Geom(M)$.  Letting $\Lambda^\gridded=(\Lambda,G)$ denote the standard gridded figure of $M$, there is a finite independent set $P\subseteq\Lambda$ such that $(P,G)$ is equivalent to some gridding of $\pi$.  Now we may add a point, say $x$, to $P$ to obtain an independent point set which is equivalent to $\tau$.  By moving $x$ without affecting the equivalence class of $P\cup\{x\}$, we may further assume that $x$ lies in the interior of a cell.

Let $h$ and $v$ denote, respectively, the horizontal and vertical lines passing through $x$, as shown on the left of Figure~\ref{fig-gridding-refinements}.  Our goal is to create a new, refined gridding of $\Lambda$ which contains $v$ and $h$ as grid lines.

The \emph{offset} of a horizontal (resp., vertical) line is the distance between that line and the base line of the row (resp., column) of $\Lambda^\gridded$ that it slices through (recall from Section~\ref{sec-words-and-encodings} that the standard gridded figure of a partial multiplication matrix has designated base lines determined by its column and row signs).  From the definition of a partial multiplication matrix, it follows that if a vertical line of offset $d$ slices a nonempty cell of $\Lambda$, it intersects the line segment in this cell precisely where the horizontal line of offset $d$ slices the line segment.

Because the cells of a standard gridded figure have unit width, $v$ and $h$ have offsets strictly between $0$ and $1$, say $0<d_1\le d_2<1$.  By possibly moving $x$ slightly without affecting the equivalence class of $P\cup\{x\}$ (which we may do because $P\cup\{x\}$ is independent), we may assume that $0<d_1<d_2<1$.  We now define our refined gridding $H$, an example of which is shown on the right of Figure~\ref{fig-gridding-refinements}.  We take $H$ to consist of the grid lines of $G$ together with the $2t$ vertical lines which pass through each column of $\Lambda$ at the offsets $d_1$ and $d_2$, and the $2u$ horizontal lines which pass through each row of $\Lambda$ at the same two offsets.  It follows from our observations above that $(\Lambda,H)$ is equivalent to the standard gridded figure of $M^{\times 3}$, and in particular consists of a grid containing line segments of slope $\pm 1$, each of which runs from corner to corner in its cell.  By possibly moving the points slightly, we may assume that no point of $P$ lies on a grid line in $H$ (although $x$ lies on two such lines).

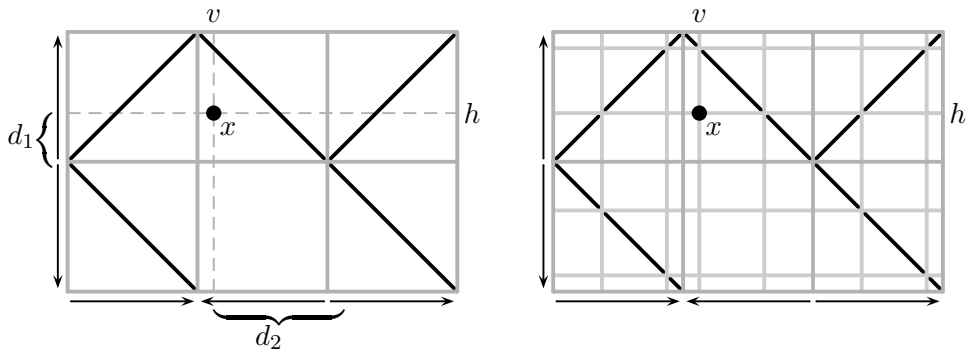
\begin{figure}
\begin{center}
\begin{tabular}{ccc}
\psset{xunit=0.017in, yunit=0.017in}
\psset{linewidth=0.005in}
\begin{pspicture}(-17,-18)(127,87)
\psline[linecolor=darkgray,linestyle=dashed,linewidth=0.01in]{c-c}(0,55)(120,55)
\psline[linecolor=darkgray,linestyle=dashed,linewidth=0.01in]{c-c}(45,0)(45,80)
\psline[linecolor=black,linestyle=solid,linewidth=0.02in](0,40)(40,0)
\psline[linecolor=black,linestyle=solid,linewidth=0.02in](0,40)(40,80)
\psline[linecolor=black,linestyle=solid,linewidth=0.02in](40,80)(80,40)
\psline[linecolor=black,linestyle=solid,linewidth=0.02in](80,40)(120,0)
\psline[linecolor=black,linestyle=solid,linewidth=0.02in](80,40)(120,80)
\psline[linecolor=darkgray,linestyle=solid,linewidth=0.02in]{c-c}(0,0)(0,80)
\psline[linecolor=darkgray,linestyle=solid,linewidth=0.02in]{c-c}(40,0)(40,80)
\psline[linecolor=darkgray,linestyle=solid,linewidth=0.02in]{c-c}(80,0)(80,80)
\psline[linecolor=darkgray,linestyle=solid,linewidth=0.02in]{c-c}(120,0)(120,80)
\psline[linecolor=darkgray,linestyle=solid,linewidth=0.02in]{c-c}(0,0)(120,0)
\psline[linecolor=darkgray,linestyle=solid,linewidth=0.02in]{c-c}(0,40)(120,40)
\psline[linecolor=darkgray,linestyle=solid,linewidth=0.02in]{c-c}(0,80)(120,80)
\pscircle*(45,55){0.04in}
\uput[-45](44,55){$x$}
\psline[linecolor=black,linestyle=solid,linewidth=0.01in,arrowsize=0.05in]{<-c}(-3,1)(-3,39)
\psline[linecolor=black,linestyle=solid,linewidth=0.01in,arrowsize=0.05in]{c->}(-3,41)(-3,79)
\psline[linecolor=black,linestyle=solid,linewidth=0.01in,arrowsize=0.05in]{c->}(1,-3)(39,-3)
\psline[linecolor=black,linestyle=solid,linewidth=0.01in,arrowsize=0.05in]{<-c}(41,-3)(79,-3)
\psline[linecolor=black,linestyle=solid,linewidth=0.01in,arrowsize=0.05in]{c->}(81,-3)(119,-3)
\rput[tl](-10,55){\rotateright{$\underbrace{\rule{0.28in}{0in}}$}}
\rput[r](-10,48){$d_1$}
\rput[tl](45,-5){$\underbrace{\rule{0.68in}{0in}}$}
\rput[tc](62,-14){$d_2$}
\rput[c](45,85){$v$}
\rput[l](122,55){$h$}
\end{pspicture}
&\rule{0in}{0pt}&
\psset{xunit=0.017in, yunit=0.017in}
\psset{linewidth=0.005in}
\begin{pspicture}(-3,-18)(127,87)
\psline[linecolor=black,linestyle=solid,linewidth=0.02in](0,40)(40,0)
\psline[linecolor=black,linestyle=solid,linewidth=0.02in](0,40)(40,80)
\psline[linecolor=black,linestyle=solid,linewidth=0.02in](40,80)(80,40)
\psline[linecolor=black,linestyle=solid,linewidth=0.02in](80,40)(120,0)
\psline[linecolor=black,linestyle=solid,linewidth=0.02in](80,40)(120,80)
\psline[linecolor=lightgray,linestyle=solid,linewidth=0.02in]{c-c}(0,25)(120,25)
\psline[linecolor=lightgray,linestyle=solid,linewidth=0.02in]{c-c}(0,55)(120,55)
\psline[linecolor=lightgray,linestyle=solid,linewidth=0.02in]{c-c}(15,0)(15,80)
\psline[linecolor=lightgray,linestyle=solid,linewidth=0.02in]{c-c}(65,0)(65,80)
\psline[linecolor=lightgray,linestyle=solid,linewidth=0.02in]{c-c}(95,0)(95,80)
\psline[linecolor=lightgray,linestyle=solid,linewidth=0.02in]{c-c}(35,0)(35,80)
\psline[linecolor=lightgray,linestyle=solid,linewidth=0.02in]{c-c}(45,0)(45,80)
\psline[linecolor=lightgray,linestyle=solid,linewidth=0.02in]{c-c}(115,0)(115,80)
\psline[linecolor=lightgray,linestyle=solid,linewidth=0.02in]{c-c}(0,5)(120,5)
\psline[linecolor=lightgray,linestyle=solid,linewidth=0.02in]{c-c}(0,75)(120,75)
\psline[linecolor=darkgray,linestyle=solid,linewidth=0.02in]{c-c}(0,0)(0,80)
\psline[linecolor=darkgray,linestyle=solid,linewidth=0.02in]{c-c}(40,0)(40,80)
\psline[linecolor=darkgray,linestyle=solid,linewidth=0.02in]{c-c}(80,0)(80,80)
\psline[linecolor=darkgray,linestyle=solid,linewidth=0.02in]{c-c}(120,0)(120,80)
\psline[linecolor=darkgray,linestyle=solid,linewidth=0.02in]{c-c}(0,0)(120,0)
\psline[linecolor=darkgray,linestyle=solid,linewidth=0.02in]{c-c}(0,40)(120,40)
\psline[linecolor=darkgray,linestyle=solid,linewidth=0.02in]{c-c}(0,80)(120,80)
\pscircle*(45,55){0.04in}
\uput[-45](44,55){$x$}
\psline[linecolor=black,linestyle=solid,linewidth=0.01in,arrowsize=0.05in]{<-c}(-3,1)(-3,39)
\psline[linecolor=black,linestyle=solid,linewidth=0.01in,arrowsize=0.05in]{c->}(-3,41)(-3,79)
\psline[linecolor=black,linestyle=solid,linewidth=0.01in,arrowsize=0.05in]{c->}(1,-3)(39,-3)
\psline[linecolor=black,linestyle=solid,linewidth=0.01in,arrowsize=0.05in]{<-c}(41,-3)(79,-3)
\psline[linecolor=black,linestyle=solid,linewidth=0.01in,arrowsize=0.05in]{c->}(81,-3)(119,-3)
\rput[c](45,85){$v$}
\rput[l](122,55){$h$}
\end{pspicture}
\end{tabular}
\end{center}
\caption{On the left is the standard gridded figure of a partial multiplication matrix, together with two additional (dashed) lines intersecting at the point $x$.  On the right is the refinement defined in the proof of Theorem~\ref{thm-geom-griddable-one-point-extension}.}
\label{fig-gridding-refinements}
\end{figure}

By cutting the figure $\Lambda$ at the lines $v$ and $h$ and moving the four resulting pieces, we can create a new column and row, whose cell of intersection contains $x$.  (One can also view this as expanding the lines $v$ and $h$ into a column and a row, respectively.)  We can then fill this cell of intersection with a line segment of slope $\pm 1$ (the choice is immaterial), running from corner to corner; clearly, $x$ can be shifted onto this line segment without affecting the equivalence class of $P\cup\{x\}$.  The resulting gridded figure consists of line segments that run from corner to corner in their cells, and is equivalent to the standard gridded figure of some partial multiplication matrix $N$ which is of size $(3t+1)\times (3u+1)$, from which it follows that $\tau\in\Geom(N)$.  As there are only finitely many such matrices, we see that $\Geom^{+1}(M)$ is contained in a finite union of geometric grid classes, and so is geometrically griddable by Proposition~\ref{prop-geom-griddable-unions}.
\end{proof}

\section{Gridded Permutations and Trace Monoids}
\label{sec-gridded-reg-lang}

In the next two sections we show that geometrically griddable classes have rational generating functions.  We divide this task into two stages.  First, using trace monoids, we show that every \emph{gridded} class $\Geom^\gridded(M)$ is in bijection with a regular language.  Then, because a permutation may have many valid griddings, the second part of our task, undertaken in the next section, is to remove this multiplicity, allowing us to enumerate the geometric grid class $\Geom(M)$ and all its subclasses.

To illustrate the issue we seek to understand in this section, we refer the reader to Figure~\ref{fig-trace-encodings}.  This example shows two different words which map to the same gridded permutation under $\bij^\gridded$.  This happens because the order in which points are consecutively inserted into \emph{independent} cells --- i.e., cells which share neither a column nor a row --- is immaterial.  On the language level, this means that letters which correspond to such cells ``commute''.  For example, in Figure~\ref{fig-trace-encodings}, the second and third letters of the words can be interchanged without changing the gridded permutation, which corresponds to placing the points $p_2$ and $p_3$ at different distances from their base points.

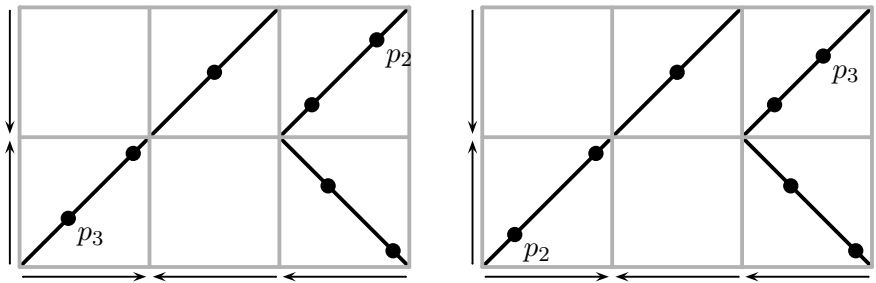
\begin{figure}
\begin{center}
\begin{tabular}{ccc}
\psset{xunit=0.017in, yunit=0.017in}
\psset{linewidth=0.005in}
\begin{pspicture}(-3,-3)(120,80)
\psline[linecolor=black,linestyle=solid,linewidth=0.02in](0,0)(40,40)
\psline[linecolor=black,linestyle=solid,linewidth=0.02in](40,40)(80,80)
\psline[linecolor=black,linestyle=solid,linewidth=0.02in](80,40)(120,0)
\psline[linecolor=black,linestyle=solid,linewidth=0.02in](80,40)(120,80)
\psline[linecolor=darkgray,linestyle=solid,linewidth=0.02in]{c-c}(0,0)(0,80)
\psline[linecolor=darkgray,linestyle=solid,linewidth=0.02in]{c-c}(40,0)(40,80)
\psline[linecolor=darkgray,linestyle=solid,linewidth=0.02in]{c-c}(80,0)(80,80)
\psline[linecolor=darkgray,linestyle=solid,linewidth=0.02in]{c-c}(120,0)(120,80)
\psline[linecolor=darkgray,linestyle=solid,linewidth=0.02in]{c-c}(0,0)(120,0)
\psline[linecolor=darkgray,linestyle=solid,linewidth=0.02in]{c-c}(0,40)(120,40)
\psline[linecolor=darkgray,linestyle=solid,linewidth=0.02in]{c-c}(0,80)(120,80)
\pscircle*(115,5){0.04in}
\pscircle*(110,70){0.04in}\uput[-45](110,70){$p_2$}
\pscircle*(15,15){0.04in}\uput[-45](15,15){$p_3$}
\pscircle*(60,60){0.04in}
\pscircle*(95,25){0.04in}
\pscircle*(90,50){0.04in}
\pscircle*(35,35){0.04in}
\psline[linecolor=black,linestyle=solid,linewidth=0.01in,arrowsize=0.05in]{c->}(-3,1)(-3,39)
\psline[linecolor=black,linestyle=solid,linewidth=0.01in,arrowsize=0.05in]{<-c}(-3,41)(-3,79)
\psline[linecolor=black,linestyle=solid,linewidth=0.01in,arrowsize=0.05in]{c->}(1,-3)(39,-3)
\psline[linecolor=black,linestyle=solid,linewidth=0.01in,arrowsize=0.05in]{<-c}(41,-3)(79,-3)
\psline[linecolor=black,linestyle=solid,linewidth=0.01in,arrowsize=0.05in]{<-c}(81,-3)(119,-3)
\end{pspicture}
&\rule{0in}{0pt}&
\psset{xunit=0.017in, yunit=0.017in}
\psset{linewidth=0.005in}
\begin{pspicture}(-3,-3)(120,80)
\psline[linecolor=black,linestyle=solid,linewidth=0.02in](0,0)(40,40)
\psline[linecolor=black,linestyle=solid,linewidth=0.02in](40,40)(80,80)
\psline[linecolor=black,linestyle=solid,linewidth=0.02in](80,40)(120,0)
\psline[linecolor=black,linestyle=solid,linewidth=0.02in](80,40)(120,80)
\psline[linecolor=darkgray,linestyle=solid,linewidth=0.02in]{c-c}(0,0)(0,80)
\psline[linecolor=darkgray,linestyle=solid,linewidth=0.02in]{c-c}(40,0)(40,80)
\psline[linecolor=darkgray,linestyle=solid,linewidth=0.02in]{c-c}(80,0)(80,80)
\psline[linecolor=darkgray,linestyle=solid,linewidth=0.02in]{c-c}(120,0)(120,80)
\psline[linecolor=darkgray,linestyle=solid,linewidth=0.02in]{c-c}(0,0)(120,0)
\psline[linecolor=darkgray,linestyle=solid,linewidth=0.02in]{c-c}(0,40)(120,40)
\psline[linecolor=darkgray,linestyle=solid,linewidth=0.02in]{c-c}(0,80)(120,80)
\pscircle*(115,5){0.04in}
\pscircle*(10,10){0.04in}\uput[-45](10,10){$p_2$}
\pscircle*(105,65){0.04in}\uput[-45](105,65){$p_3$}
\pscircle*(60,60){0.04in}
\pscircle*(95,25){0.04in}
\pscircle*(90,50){0.04in}
\pscircle*(35,35){0.04in}
\psline[linecolor=black,linestyle=solid,linewidth=0.01in,arrowsize=0.05in]{c->}(-3,1)(-3,39)
\psline[linecolor=black,linestyle=solid,linewidth=0.01in,arrowsize=0.05in]{<-c}(-3,41)(-3,79)
\psline[linecolor=black,linestyle=solid,linewidth=0.01in,arrowsize=0.05in]{c->}(1,-3)(39,-3)
\psline[linecolor=black,linestyle=solid,linewidth=0.01in,arrowsize=0.05in]{<-c}(41,-3)(79,-3)
\psline[linecolor=black,linestyle=solid,linewidth=0.01in,arrowsize=0.05in]{<-c}(81,-3)(119,-3)
\end{pspicture}
\end{tabular}
\end{center}
\caption{Two drawings of a particular gridding of the permutation $2465371$.  The drawing on the left is encoded as $\cell{3}{1}\cell{3}{2}\cell{1}{1}\cell{2}{2}\cell{3}{1}\cell{3}{2}\cell{1}{1}$ while the drawing on the right is encoded as $\cell{3}{1}\cell{1}{1}\cell{3}{2}\cell{2}{2}\cell{3}{1}\cell{3}{2}\cell{1}{1}$.}
\label{fig-trace-encodings}
\end{figure}

More precisely, suppose that $M$ is a partial multiplication matrix with cell alphabet $\Sigma$.  We say that the two words $v,w\in\Sigma^*$ are \emph{equivalent} if one can be obtained from the other by successively interchanging adjacent letters which represent independent cells.  The equivalence classes of this relation therefore form a \emph{trace monoid}, which can be defined by the presentation
$$
\langle\Sigma \:|\: \mbox{$a_{ij}a_{k\ell}=a_{k\ell}a_{ij}$ whenever $i\neq k$ and $j\neq \ell$}\rangle.
$$
An element of this monoid (an equivalence class of words in $\Sigma^\ast$) is called a \emph{trace}; it is an elementary and foundational fact about presentations that two words lie in the same trace if and only if one can be obtained from the other by a finite sequence of applications of the defining relations $a_{ij}a_{k\ell}=a_{k\ell}a_{ij}$ whenever $i\neq k$ and $j\neq \ell$ (Howie~\cite[Proposition 1.5.9 and Section 1.6]{howie:fundamentals-of:}).  The theory of trace monoids is understood to considerable depth, see for example Diekert~\cite{diekert:combinatorics-o:}.

\begin{proposition}\label{prop-gridded-encoding-trace}
Let $M$ be a partial multiplication matrix with cell alphabet $\Sigma$.  For words $v,w\in\Sigma^\ast$, we have $\bij^\gridded(v)=\bij^\gridded(w)$ if and only if $v$ and $w$ have the same trace in the trace monoid of $M$.
\end{proposition}
\begin{proof}
It is clear from the preceding discussion that $\bij^\gridded(v)=\bij^\gridded(w)$ whenever $v$ and $w$ have the same trace.

Now suppose that $\bij^\gridded(v)=\bij^\gridded(w)$.  We aim to prove that $v$ and $w$ have the same trace in the trace monoid of $M$.  Clearly $v$ and $w$ must have the same length, say $n$.  We prove our assertion by induction on $n$.  In the case where $n=1$, note that $\bij^\gridded(a_{k,\ell})$ consists of a single point in cell $C_{k,\ell}$, so if $|v|=|w|=1$, then $\bij^\gridded(v)=\bij^\gridded(w)$ implies that $v=w$, and thus $v$ and $w$ are in the same trace trivially.  Thus we assume that $n\ge 2$ and that the assertion is true for all words of length $n-1$.  

Write $v=v_1\cdots v_n$ and $w=w_1\cdots w_n$.  If $v_1=w_1$ then from the definition of $\bij^\gridded$ we have $\bij^\gridded(v_2\cdots v_n)=\bij^\gridded(w_2\cdots w_n)$, and the assertion follows by induction.  Therefore we may assume that $v_1\neq w_1$.

Suppose that the column and row signs of $M$ are $c_1,\dots,c_t$ and $r_1,\dots,r_u$ and that $v_1=a_{k\ell}$.  While there are four cases depending on the column and row signs $c_k$ and $r_\ell$, they are essentially identical, so we assume that $c_k=r_\ell=1$.  This means that among the points created by following the definition of $\bij^\gridded$, the point corresponding to $v_1$ in $\bij^\gridded(v)$ is the leftmost point in its column and bottommost point in its row.  Since $\bij^\gridded(v)=\bij^\gridded(w)$ as \emph{gridded} permutations, $\bij^\gridded(w)$ must also contain a point in cell $M_{k,\ell}$ which is the leftmost point in its column and bottommost point in its row; suppose this point corresponds to $w_i$.  Because $c_k=1$, the entries in column $k$ are placed from left to right, so $w_i$ must be the first letter of $w$ which corresponds to a cell in column $k$.  Similarly, $w_i$ must be the first letter of $w$ which corresponds to a cell in row $\ell$.  Therefore all of the letters $w_1,\dots,w_{i-1}$ correspond to cells which are in different rows and columns from $M_{k,\ell}$.  This shows that $w$ lies in the same trace as the word
$$
w'=w_i w_1\cdots w_{i-1} w_{i+1}\cdots w_n=v_1w_1\cdots w_{i-1} w_{i+1}\cdots w_n.
$$
Observe that $\bij^\gridded(w')=\bij^\gridded(w)$, and thus $\bij^\gridded(w')=\bij^\gridded(v)$.  Finally, because $v$ and $w'$ have the same first letter, we see that $v$ and $w'$ are in the same trace by the first case of this argument, which implies that $v$ and $w$ are in the same trace, proving the proposition.
\end{proof}

This result established, we are reduced to the task of choosing from each trace a unique representative, which is a well-understood problem.

\begin{proposition}
[Anisimov and Knuth; see Diekert~{\cite[Corollary 1.2.3]{diekert:combinatorics-o:}}]
\label{prop-trace-regular}
In any trace monoid, it is possible to choose a unique representative from each trace in such a way that the resulting set of representatives forms a regular language.
\end{proposition}

We immediately obtain the following.

\begin{corollary}
\label{cor-geometric-gridded-reg}
For every partial multiplication matrix $M$, the (gridded) class $\Geom^\gridded(M)$ is in bijection with a regular language.
\end{corollary}

\section{Regular Languages and Rational Generating Functions}
\label{sec-geom-reg-lang}

We now shift our attention to the \emph{ungridded} permutations in a geometrically griddable class.  This is the most technical argument of the paper, and we outline the general approach before delving into the formalisation.

The crux of the issue relates to the ungridded version of the encoding map $\bij$ and arises because it is possible that $\bij(v)=\bij(w)$ for two words $v,w\in\Sigma^\ast$ even though $\bij^\gridded(v)\neq \bij^\gridded(w)$.  This happens precisely when $\pi=\bij(v)=\bij(w)$ admits two different griddings.  To discuss these different griddings, we say that the gridded permutation $(\pi,G)\in\Geom^\gridded(M)$ is a \emph{$\Geom^\gridded(M)$ gridding} of the permutation $\pi\in\Geom(M)$.  Since our goal is to establish a bijection between any geometrically griddable class $\C$ and a regular language, if a permutation in $\C$ has multiple $\Geom^\gridded(M)$ griddings we must choose only one.  To do so, we introduce a total order on the set of all such griddings of a fixed permutation $\pi$, and aim to keep only those which are minimal in this order.

The problem is thus translated to that of recognising minimal $\Geom^\gridded(M)$ griddings: Given a word $w\in\Sigma^\ast$, how can we determine if the gridded permutation $\bij^\gridded(w)$ represents the minimal $\Geom^\gridded(M)$ gridding of $\pi=\bij(w)$?  If not, then there is a lesser gridding of $\pi$, given by some $\bij^\gridded(v)$.  The fact that $\bij^\gridded(v)$ is less than $\bij^\gridded(w)$ is witnessed by the position of one or more particular entries which lie in different cells in $\bij^\gridded(v)$ and $\bij^\gridded(w)$.

In order to discuss these entries, we use the terminology of marked permutations and marked words.  A \emph{marked permutation} is a permutation in which the entries are allowed to be marked, which we designate with an overline.  (Thus our marked permutations look like some other authors' signed permutations, although the marking is meant to convey a completely different concept.)  The containment order extends naturally, whereby we make sure that the markings line up.  Formally we say that the marked permutation $\pi$ of length $n$ contains the marked permutation $\sigma$ of length $k$ if $\pi$ has a subsequence $\pi(i_1),\pi(i_2),\dots,\pi(i_k)$ such that:
\begin{itemize}
\item $\pi(i_1)\pi(i_2)\cdots\pi(i_k)$ is order isomorphic to $\sigma$ (this is the standard containment order), and
\item for each $1\le j\le k$, $\pi(i_j)$ is marked if and only if $\sigma(j)$ is marked.
\end{itemize}
For example, $\pi=\overline{3}\overline{9}1\overline{8}\overline{6}7\overline{4}\overline{5}2$ contains $\sigma=\overline{5}1\overline{3}42$, as can be seen by considering the subsequence $\overline{9}1\overline{6}72$.  A \emph{marked permutation class} is then a set of marked permutations which is closed downward under this containment order.

The mappings $\bij$ and $\bij^\gridded$ can be extended in the obvious manner to order-preserving mappings $\overline{\bij}$ and $\overline{\bij}^\gridded$ from $(\Sigma\cup\overline{\Sigma})^\ast$ to the marked version of $\Geom(M)$, without and with grid lines, respectively, where here $\overline{\Sigma}$ is the \emph{marked cell alphabet} $\{\overline{a}\st a\in\Sigma\}$, and both mappings send marked letters to marked entries.

\begin{theorem}
\label{thm-geom-griddable-reg-lang}
Every geometrically griddable class is in bijection with a regular language, and thus has a rational generating function.
\end{theorem}
\begin{proof}
Let $\C$ be a geometrically griddable class.  By Proposition~\ref{prop-geom-pmm}, $\C\subseteq\Geom(M)$ for a partial multiplication matrix $M$.  By Corollary~\ref{cor-geometric-gridded-reg}, there is a regular language, say $L^\gridded$, such that $\bij^\gridded\st L^\gridded\rightarrow\Geom^\gridded(M)$ is a bijection.

We begin by defining a total order on the various griddings of each permutation $\pi\in\Geom(M)$, and thus also on the $\Geom^\gridded(M)$ griddings of permutations in $\C$.  Roughly, this order prefers griddings in which the entries of $\pi$ lie in cells as far to the left and bottom as possible, or, in terms of grid lines, the order prefers griddings in which the grid lines lie as far to the right and top as possible.  Suppose we have two different $\Geom^\gridded(M)$ griddings of $\pi$ given by the grid lines $G$ and $H$.  Because these griddings are different, there must be a leftmost vertical grid line, or failing that, a bottommost horizontal grid line, which moved.  Note that in the former case, $G$ and $H$ will contain a different number of entries in the corresponding column, while in the latter case they will contain a different number of entries in the corresponding row.

Formally, if $(\pi,G)$ contains the same number of entries as $(\pi,H)$ in each of the leftmost $k-1$ columns, but contains more entries than $(\pi,H)$ in column $k$, then we write $(\pi,G)\gridlex (\pi,H)$, and we say that column $k$ \emph{witnesses} this fact.  Otherwise, if $(\pi,G)$ and $(\pi,H)$ contain the same number of entries in each column and in each of the bottom $\ell-1$ rows, but $(\pi,G)$ contains more entries than $(\pi,H)$ in row $\ell$, then $(\pi,G)\gridlex (\pi,H)$, and we say that row $\ell$ \emph{witnesses} this fact.  Given a permutation $\pi\in\C$, our aim is to choose, from all $\Geom^\gridded(M)$ griddings of $\pi$, the minimal one.

Given a triple $(\pi,G,H)$ where $(\pi,G)$ and $(\pi,H)$ are $\Geom^\gridded(M)$ griddings of some $\pi\in\C$ with $(\pi,G)\gridlexeq(\pi,H)$, we mark the entries of the gridded permutation $(\pi,H)$ in the following manner.  If $(\pi,G)=(\pi,H)$, then no markings are applied.  If $(\pi,G)\gridlex (\pi,H)$ is witnessed by column $k$ then we mark those entries of $\pi$ which lie in column $k$ in $(\pi,G)$ but not in $(\pi,H)$.  Otherwise, if $(\pi,G)\gridlex (\pi,H)$ is witnessed by row $\ell$ then we similarly mark those entries of $\pi$ which lie in row $\ell$ in $(\pi,G)$ but not in $(\pi,H)$.

Let $\overline{\C}^\gridded$ denote the set of all marked gridded permutations obtained from triples $(\pi,G,H)$, $\pi\in\C$, in this manner.  Because $\Geom^\gridded(M)$ griddings are inherited by subpermutations and $\overline{\bij}^\gridded$ is order-preserving, the language
$$
\overline{J}
=
\left(\overline{\bij}^\gridded\right)^{-1}\left(\overline{\C}^\gridded\right)
$$
is subword-closed in $\left(\Sigma\cup\overline{\Sigma}\right)^\ast$.  In particular, $\overline{J}$ is a regular language.  Loosely speaking, the words in $\overline{J}$ containing marked letters carry information about all the non-minimal griddings.  Our goal is therefore to recognise these non-minimal griddings in $\Sigma^\ast$, and remove them from the language $L^\gridded$.

Consider any word $w\in L^\gridded$, which encodes the $\Geom^\gridded(M)$-gridded permutation $\bij^\gridded(w)=(\pi,H)$.  The gridding given by $H$ is \emph{not} the minimal $\Geom^\gridded(M)$ gridding of $\pi$ precisely if $\overline{J}$ contains a copy of $w$ with one or more marked letters.  Let $\Gamma:\left(\Sigma\cup\overline{\Sigma}\right)^\ast\to\Sigma^\ast$ denote the homomorphism which removes markings, i.e., the homomorphism given by $a,\overline{a}\mapsto a$.  The words which represent non-minimal griddings (precisely the words we wish to remove from $L^\gridded$) are therefore the set
$$
K=\Gamma\left(
J\cap
\left(\left(\Sigma\cup\overline{\Sigma}\right)^\ast\setminus\Sigma^\ast\right)
\right).
$$
By the basic properties of regular languages, it can be seen that $K$ and hence $L=L^\gridded\setminus K$ are regular.  The proof is then complete as $\bij:L\rightarrow\C$ is a bijection.
\end{proof}

\section{Indecomposable and Simple Permutations}
\label{sec-geom-simples}

Here we adapt the techniques of the previous section to establish bijections between regular languages and three structurally important subsets of geometrically griddable classes.

An \emph{interval} in the permutation $\pi$ is a set of contiguous indices $I=[a,b]$ such that the set of values $\pi(I)=\{\pi(i) : i\in I\}$ is also contiguous.  Every permutation of length $n$ has trivial intervals of lengths $0$, $1$, and $n$; the permutation $\pi$ of length at least $2$ is said to be \emph{simple} if it has no other intervals.  The importance of simple permutations in the study of permutation classes has been recognised since Albert and Atkinson~\cite{albert:simple-permutat:}, whose terminology we follow; we refer to Brignall~\cite{brignall:a-survey-of-sim:} for a recent survey.

Given a permutation $\sigma$ of length $m$ and nonempty permutations $\alpha_1,\dots,\alpha_m$, the \emph{inflation} of $\sigma$ by $\alpha_1,\dots,\alpha_m$ --- denoted $\sigma[\alpha_1,\dots,\alpha_m]$ --- is the permutation obtained by replacing each entry $\sigma(i)$ by an interval that is order isomorphic to $\alpha_i$ in such a way that the intervals are order isomorphic to $\sigma$.  For example,
$$
2413[1,132,321,12]=4\ 798\ 321\ 56.
$$

Two particular types of inflations have been given their own names.  These are the \emph{direct sum}, or simply \emph{sum}, $\pi\oplus\sigma=12[\pi,\sigma]$ and the \emph{skew sum} $\pi\ominus\sigma=21[\pi,\sigma]$.  We say that a permutation is \emph{sum indecomposable} if it is not the sum of two shorter permutations and \emph{skew indecomposable} if it is not the skew sum of two shorter permutations.

These notions defined, we are ready to construct regular languages which encode the simple, sum indecomposable, and skew indecomposable permutations in a geometrically griddable class.  Our approach mirrors the proof of Theorem~\ref{thm-geom-griddable-reg-lang}.

\begin{theorem}
\label{thm-geom-simple-reg-lang}
The simple, sum indecomposable, and skew indecomposable permutations in every geometrically griddable class are each in bijection with a regular language, and thus have rational generating functions.
\end{theorem}
\begin{proof}
We give complete details only for the case of simple permutations; the minor modifications needed to handle sum indecomposable and skew indecomposable permutations are indicated at the end of the proof.

Let $\C$ be a geometrically griddable class.  As usual, Proposition~\ref{prop-geom-pmm} implies that $\C\subseteq\Geom(M)$ for a partial multiplication matrix $M$.  By Theorem~\ref{thm-geom-griddable-reg-lang}, there is a regular language $L$ such that $\bij:L\to\C$ is a bijection.

Let $\overline{\C}$ denote the set of all permutations in $\C$ with all possible markings.  We say that the markings of a marked permutation are \emph{interval consistent} if the marked entries of the permutation form a (possibly trivial) interval.  Let $\overline{\mathcal{I}}$ consist of all marked permutations in $\overline{\C}$ with interval consistent markings.  Because $\overline{\bij}$ is order-preserving, the preimage
$$
\overline{J}=\overline{\bij}^{-1}(\overline{\mathcal{I}})
$$
is subword-closed in $\left(\Sigma\cup\overline{\Sigma}\right)^\ast$, and thus is a regular language.

Now consider any permutation $\pi\in\C$.  This permutation is simple if and only if it does not have a nontrivial interval.  In terms of our markings, therefore, $\pi$ is simple if and only if there is no interval consistent marking of $\pi$ which contains at least two marked entries and at least one unmarked entry.  On the language level, a word over $\left(\Sigma\cup\overline{\Sigma}\right)^\ast$ has at least two marked entries and at least one unmarked entry precisely if it lies in
$$
\left(\Sigma\cup\overline{\Sigma}\right)^\ast
\setminus
\left(
\Sigma^\ast
\cup
\overline{\Sigma}^\ast
\cup
\Sigma^\ast\overline{\Sigma}\Sigma^\ast
\right).
$$
Therefore the words in $\Sigma^\ast$ which represent non-simple permutations in $\C$ are precisely those in the set
$$
K=\Gamma\left(
\overline{J}
\cap
\left(
\left(\Sigma\cup\overline{\Sigma}\right)^\ast
\setminus
\left(
\Sigma^\ast
\cup
\overline{\Sigma}^\ast
\cup
\Sigma^\ast\overline{\Sigma}\Sigma^\ast
\right)
\right)
\right),
$$
where, as in the proof of Theorem~\ref{thm-geom-griddable-reg-lang}, $\Gamma:\left(\Sigma\cup\overline{\Sigma}\right)^\ast\to\Sigma^\ast$ denotes the homomorphism which removes markings.  The simple permutations in $\C$ are therefore encoded by the regular language $L\setminus K$, completing the proof of that case.

This proof can easily be adapted to the case of sum (resp., skew) indecomposable permutations by defining markings to be \emph{sum consistent} (resp., \emph{skew consistent}) if the underlying permutation is the sum (resp., skew sum) of its marked entries and its unmarked entries (in either order).
\end{proof}

\section{Atomic Decompositions}
\label{sec-atomic-decompositions}

The intersection of two geometrically griddable classes is trivially geometrically griddable, and as we observed in Proposition~\ref{prop-geom-griddable-unions}, their union is geometrically griddable as well.  Therefore, within the lattice of permutation classes, the collection of geometrically griddable classes forms a sublattice.  In this section we consider geometrically griddable classes from a lattice-theoretic viewpoint.


The permutation class $\C$ is \emph{join-irreducible} (in the usual lattice-theoretic sense) if $\C\neq\D\cup\E$ for two proper subclasses $\D,\E\subsetneq\C$.  In deference to existing literature on permutation classes, we refer to join-irreducible classes as \emph{atomic}.  It is not difficult to show that the {\it joint embedding property\/} is a necessary and sufficient condition for the permutation class $\C$ to be atomic; this condition states that for all $\pi,\sigma\in\C$, there is a $\tau\in\C$ containing both $\pi$ and $\sigma$.

Fra{\"{\i}}ss\'e~\cite{fraisse:sur-lextension-:} studied atomic classes in the more general context of relational structures, and established another necessary and sufficient condition.  Specialised to the context of permutations, given two linearly ordered sets $A$ and $B$ and a bijection $f:A\rightarrow B$, every finite subset $\{a_1<\cdots<a_n\}\subseteq A$ maps to a finite sequence $f(a_1),\dots,f(a_n)\in B$ that is order isomorphic to a unique permutation.  We call the set of permutations that arise in this manner the {\it age of $f$\/}, denoted $\Age(f:A\rightarrow B)$.  A proof of the following result in the language of permutations can also be found in Atkinson, Murphy and Ru\v{s}kuc~\cite{atkinson:pattern-avoidan:}.

\begin{theorem}[Fra{\"{\i}}ss{\'e}~\cite{fraisse:sur-lextension-:}; see also Hodges~{\cite[Section 7.1]{hodges:model-theory:}}]\label{atomic-tfae}
The following three conditions are equivalent for a permutation class $\C$:
\begin{enumerate}
\item[(1)] $\C$ is atomic,
\item[(2)] $\C$ satisfies the joint embedding property, and
\item[(3)] $\C=\Age(f:A\rightarrow B)$ for a bijection $f$ between two countable linear orders $A$ and $B$.
\end{enumerate}
\end{theorem}

The next proposition is a specialisation of standard lattice-theoretic facts which may be found in more general terms in many sources, such as Birkhoff~\cite{birkhoff:lattice-theory:}.

\begin{proposition}\label{pwo-atomic-union}
Every pwo permutation class can be expressed as a finite union of atomic classes.
\end{proposition}

In order to describe the atomic geometrically griddable classes as $\Geom(M)$ for certain matrices $M$, we allow our matrices to contain entries equal to $\bullet$, to signify cells in which a permutation may contain at most one point.  We have to be a bit careful here, as it is unclear how one should interpret $\fnmatrix{rr}{\bullet&\bullet}$.  We simply forbid such configurations, in the sense formalised by the following definitions.

Suppose that $M$ is a $\zdpm$ matrix, meaning that each entry of $M$ lies in $\{0,\bullet,1,-1\}$.  We say that $M$ is $\bullet$-isolated if every $\bullet$ entry is the only nonzero entry in its column and row.  Given a $\bullet$-isolated $\zdpm$ matrix $M$, its \emph{standard figure} is the point set in $\mathbb{R}^2$ consisting of:
\begin{itemize}
\item a single point at $(k-\nicefrac{1}{2},\ell-\nicefrac{1}{2})$ if $M_{k,\ell}=\bullet$,
\item the line segment from $(k-1,\ell-1)$ to $(k,\ell)$ if $M_{k,\ell}=1$, or
\item the line segment from $(k-1,\ell)$ to $(k,\ell-1)$ if $M_{k,\ell}=-1$.
\end{itemize}
We can then extend the notion of geometric grid classes to $\zdpm$ matrices in the obvious manner, and we have the following result.

\begin{theorem}
\label{thm-geom-griddable-union-atomics}
The atomic geometrically griddable classes are precisely the geometric grid classes of $\bullet$-isolated $\zdpm$ matrices, and every geometrically griddable class can be expressed as a finite union of such classes.
\end{theorem}
\begin{proof}
First suppose that $M$ is a $\bullet$-isolated $\zdpm$ matrix.  It is clear from the geometric description of $\Geom(M)$ that given any two permutations $\pi,\sigma\in\Geom(M)$, there is a permutation $\tau\in\Geom(M)$ such that $\tau\ge\pi,\sigma$, so such classes satisfy the joint embedding property and are thus atomic by Theorem~\ref{atomic-tfae}.

Next we show that every geometrically griddable class can be expressed as a finite union of classes of the form $\Geom(M)$ where $M$ is a $\bullet$-isolated $\zdpm$ matrix.  Note that this will imply that the only atomic geometric griddable classes are of the latter type.

Let $\C$ be a geometrically griddable class.  By Proposition~\ref{prop-geom-pmm}, $\C\subseteq\Geom(M)$ for some partial multiplication matrix $M$ (note that $M$ is a $\zpm$ matrix) with cell alphabet $\Sigma$.  Since the encoding map $\bij\st\Sigma^\ast\rightarrow\Geom(M)$ is order-preserving (Proposition~\ref{prop-properties-of-bij}), the preimage $\bij^{-1}(\C)$ is a subword-closed language.  By Proposition~\ref{prop-subword-closed-decomp}, we know that $\bij^{-1}(\C)$ is a finite union of languages of the form
\begin{equation}\label{eqn-subword-form}
\Sigma_1^\ast \{\emptyword, a_2\} \Sigma_3^\ast\{\emptyword, a_4\} \ldots \Sigma_{2q}^\ast \{\emptyword, a_{2q}\} \Sigma_{2q+1}^\ast
\tag{$\dagger$}
\end{equation}
where $q\geq 0$, $\Sigma_1,\ldots,\Sigma_{2q+1}\subseteq \Sigma$, and $a_2,\ldots ,a_{2q}\in \Sigma$.

Let $L$ denote an arbitrary language of the form \eqref{eqn-subword-form}. We will show that $\bij(L)=\Geom(M_L)$ for some $\bullet$-isolated $\zdpm$ matrix $M_L$, from which the result will follow.  We start with the standard gridded figure $\Lambda^\gridded=(\Lambda,G)$ of $M^{\times (2q+1)}$.  Recall that each cell of the standard gridded figure of $M$ becomes $(2q+1)^2$ cells in $\Lambda^\gridded$ of which $(2q+1)$ are nonempty; we use this to label the nonempty cells of $\Lambda^\gridded$ by $C_{k,\ell}^{(s)}$ for $s\in[2q+1]$, in order of increasing distance from the base point as it would be in the standard gridded figure of $M$.

The permutations in $\bij(L)$ are then equivalent to finite independent sets $P\subseteq\Lambda$ of the following form:
\begin{itemize}
\item For odd $s\in[2q+1]$, $P$ may contain any points of $\Lambda$ belonging to cells $C_{k,\ell}^{(s)}$ for any $k$, $\ell$ such that $a_{k,\ell}\in \Sigma_s$, and no points from other cells.
\item For even $s\in[2q+1]$, $P$ may contain at most one point from the cell $C_{k,\ell}^{(s)}$ where $a_{k,\ell}=a_s$, and no points from any other cells.
\end{itemize}

\begin{figure}
\begin{center}
\begin{tabular}{ccc}
\psset{xunit=0.017in, yunit=0.017in}
\psset{linewidth=0.005in}
\begin{pspicture}(-3,-3)(120,80)
\psline[linecolor=black,linestyle=solid,linewidth=0.02in](0,40)(40,0)
\psline[linecolor=black,linestyle=solid,linewidth=0.02in](0,40)(40,80)
\psline[linecolor=black,linestyle=solid,linewidth=0.02in](40,80)(80,40)
\psline[linecolor=black,linestyle=solid,linewidth=0.02in](80,40)(120,80)
\psline[linecolor=darkgray,linestyle=solid,linewidth=0.02in]{c-c}(0,0)(0,80)
\psline[linecolor=darkgray,linestyle=solid,linewidth=0.02in]{c-c}(40,0)(40,80)
\psline[linecolor=darkgray,linestyle=solid,linewidth=0.02in]{c-c}(80,0)(80,80)
\psline[linecolor=darkgray,linestyle=solid,linewidth=0.02in]{c-c}(120,0)(120,80)
\psline[linecolor=darkgray,linestyle=solid,linewidth=0.02in]{c-c}(0,0)(120,0)
\psline[linecolor=darkgray,linestyle=solid,linewidth=0.02in]{c-c}(0,40)(120,40)
\psline[linecolor=darkgray,linestyle=solid,linewidth=0.02in]{c-c}(0,80)(120,80)
\psline[linecolor=black,linestyle=solid,linewidth=0.01in,arrowsize=0.05in]{<-c}(-3,1)(-3,39)
\psline[linecolor=black,linestyle=solid,linewidth=0.01in,arrowsize=0.05in]{c->}(-3,41)(-3,79)
\psline[linecolor=black,linestyle=solid,linewidth=0.01in,arrowsize=0.05in]{c->}(1,-3)(39,-3)
\psline[linecolor=black,linestyle=solid,linewidth=0.01in,arrowsize=0.05in]{<-c}(41,-3)(79,-3)
\psline[linecolor=black,linestyle=solid,linewidth=0.01in,arrowsize=0.05in]{c->}(81,-3)(119,-3)
\end{pspicture}
&\rule{0in}{0pt}&
\psset{xunit=0.017in, yunit=0.017in}
\psset{linewidth=0.005in}
\begin{pspicture}(-3,-3)(120,80)
%
\psline[linecolor=black,linestyle=solid,linewidth=0.02in](0,40)(8,32)
\psline[linecolor=black,linestyle=solid,linewidth=0.02in](16,24)(24,16)
\psline[linecolor=black,linestyle=solid,linewidth=0.02in](0,40)(8,48)
\psline[linecolor=black,linestyle=solid,linewidth=0.02in](16,56)(24,64)
\psline[linecolor=black,linestyle=solid,linewidth=0.02in](32,72)(40,80)
\pscircle*(52,68){0.04in}
\pscircle*(68,52){0.04in}
\psline[linecolor=black,linestyle=solid,linewidth=0.02in](96,56)(104,64)
\psline[linecolor=black,linestyle=solid,linewidth=0.02in](112,72)(120,80)
\psline[linecolor=darkgray,linestyle=solid,linewidth=0.02in]{c-c}(0,0)(0,80)
\psline[linecolor=darkgray,linestyle=solid,linewidth=0.02in]{c-c}(40,0)(40,80)
\psline[linecolor=darkgray,linestyle=solid,linewidth=0.02in]{c-c}(80,0)(80,80)
\psline[linecolor=darkgray,linestyle=solid,linewidth=0.02in]{c-c}(120,0)(120,80)
\psline[linecolor=darkgray,linestyle=solid,linewidth=0.02in]{c-c}(0,0)(120,0)
\psline[linecolor=darkgray,linestyle=solid,linewidth=0.02in]{c-c}(0,40)(120,40)
\psline[linecolor=darkgray,linestyle=solid,linewidth=0.02in]{c-c}(0,80)(120,80)
\multips(0,0)(0,8){10}{%
	\psline[linecolor=darkgray,linestyle=solid,linewidth=0.01in]{c-c}(0,0)(120,0)
}
\multips(0,0)(8,0){15}{%
	\psline[linecolor=darkgray,linestyle=solid,linewidth=0.01in]{c-c}(0,0)(0,80)
}
\multips(0,0)(0,8){5}{%
	\psline[linecolor=black,linestyle=solid,linewidth=0.01in,arrowsize=0.05in]{<-c}(-3,1)(-3,7)
}
\multips(0,0)(0,8){5}{%
	\psline[linecolor=black,linestyle=solid,linewidth=0.01in,arrowsize=0.05in]{c->}(-3,41)(-3,47)
}
\multips(0,0)(8,0){5}{%
	\psline[linecolor=black,linestyle=solid,linewidth=0.01in,arrowsize=0.05in]{c->}(1,-3)(7,-3)
}
\multips(0,0)(8,0){5}{%
	\psline[linecolor=black,linestyle=solid,linewidth=0.01in,arrowsize=0.05in]{<-c}(41,-3)(47,-3)
}
\multips(0,0)(8,0){5}{%
	\psline[linecolor=black,linestyle=solid,linewidth=0.01in,arrowsize=0.05in]{c->}(81,-3)(87,-3)
}
\end{pspicture}
\end{tabular}
\end{center}
\caption[]{The standard gridded figure of the matrix $\fnmatrix{rrr}{1&-1&1\\-1&0&0}$ is shown on the left, while the figure on the right displays the subfigure for the subclass encoded by the language $\{\cell{1}{1}, \cell{1}{2}\}^\ast \{\emptyword, \cell{2}{2}\} \{\cell{1}{1}, \cell{1}{2}, \cell{3}{2}\}^\ast \{\emptyword, \cell{2}{2}\} \{\cell{1}{2}, \cell{3}{2}\}^\ast$.}
\label{fig-atomic-decomp-lang}
\end{figure}

Thus $\bij(L)=\Sub(\Lambda_L)$ where $\Lambda_L\subseteq \Lambda$ consists of:
\begin{itemize}
\item[(1)] all line segments of $\Lambda$ in the cells $C_{k,\ell}^{(s)}$ where $s\in[2q+1]$ is odd and $a_{k,\ell}\in\Sigma_s$, and
\item[(2)] the centre point of the subcell $C_{k,\ell}^{(s)}$ where $s\in[2q+1]$ is even and $a_s=a_{k,\ell}$.
\end{itemize}
Figure~\ref{fig-atomic-decomp-lang} shows an example of this construction.  The subfigure $\Lambda_L$ is clearly the standard gridded figure of some $\zdpm$ matrix $M_L$.  Moreover, as $\bullet$ entries can only arise from case (2) above, it follows that $M_L$ is $\bullet$-isolated, completing the proof.
\end{proof}

\section{Concluding Remarks}
\label{sec-concluding-remarks}

We have provided a comprehensive toolbox of results applicable to geometrically griddable classes, so perhaps the most immediate question is: how can one tell if a permutation class is geometrically griddable?  Huczynska and Vatter~\cite{huczynska:grid-classes-an:} have shown that a class is contained in a monotone grid class (i.e., it is \emph{griddable}) if and only if it does not contain arbitrarily long sums of $21$ or skew sums of $12$.  However, $\Grid(M)\neq\Geom(M)$ when $M$ is not a forest, so it remains to determine the precise border between griddability and \emph{geometric} griddability.

None of the major proofs in the preceding sections are effective, in as much as they all appeal to the finiteness of certain antichains of words, which follows nonconstructively from Higman's Theorem.  Therefore these proofs do not provide algorithms to accomplish any of the following:
\begin{itemize}
\item Given a $\zpm$ matrix $M$, compute the basis of $\Geom(M)$.  In particular, any bound on the length of the basis elements would provide such an algorithm.
\item Given a $\zpm$ matrix $M$, compute the generating function for any of: $\Geom(M)$, the simple permutations in $\Geom(M)$, etc.
\item Given a $\zpm$ matrix $M$ and a finite set of permutations $B$, determine the atomic decomposition of $\Geom(M)\cap\Av(B)$, and/or its enumeration.
\end{itemize}

An intriguing, and somewhat different, question is the membership problem.  Given a $\zpm$ matrix $M$, how efficiently (as a function of $n$) can one determine if a permutation of length $n$ lies in $\Geom(M)$?  Because geometric grid classes are finitely based, this problem is guaranteed to be polynomial-time, but it could conceivably be linear-time.  Such a result would extend the parallel between geometric grid classes and subword-closed languages, because the latter (and indeed all regular languages) have linear-time membership problems.  

While we believe that geometric grid classes play a special role in the structural theory of permutation classes, their non-geometric counterparts also present many natural questions.  Perhaps the most natural is the finite basis question.  Does the class $\Grid(M)$ have a finite basis for every $\zpm$ matrix $M$?  We feel that the answer should be ``yes'', but have scant evidence.  In his thesis, Waton \cite{waton:on-permutation-:} proves that the grid class
$$
\Grid\fnmatrix{rr}{1&1\\1&1}
$$
is finitely based.

Another example of a finitely based non-geometric grid class appears in one of the earliest papers on permutation patterns, where Stankova \cite{stankova:forbidden-subse:} proves that the class of permutations which can be expressed as the union of an increasing and a decreasing subsequence, called the \emph{skew-merged permutations}, has the basis $\{2143, 3412\}$.  In our notation, this class is
$$
\Grid\fnmatrix{rr}{-1&1\\1&-1}.
$$

The class of skew-merged permutations is also notable because it is the only non-geometric grid class with a known generating function,
$$
\frac{1-3x}{(1-2x)\sqrt{1-4x}},
$$
due to Atkinson~\cite{atkinson:permutations-wh:}.  Could it be the case that all (monotone) grid classes have algebraic generating functions?  A first step in this direction might be a more structural derivation of the generating function for skew-merged permutations.

\bibliographystyle{acm}
\bibliography{../refs}

\end{document}